\documentclass[12pt]{article}
 
\usepackage[margin=1in]{geometry} 
\usepackage{amsmath,amsthm,amssymb}
\usepackage{comment}

\usepackage{xcolor}
\usepackage{hyperref}
\hypersetup{
    colorlinks,
    linkcolor={blue!90!black},
    citecolor={blue!100!black},
    urlcolor={blue!80!black}
}

\usepackage{authblk}
\title{Weak pattern convergence for SLOPE and its robust versions}
\author[1]{Ivan Hejný}
\author[1]{Jonas Wallin}
\author[1,2]{Małgorzata Bogdan}
\affil[1]{Department of Statistics, Lund University}
\affil[2]{Institute of Mathematics, University of Wroclaw}
\date{}

\usepackage{titlesec}
\titleformat*{\section}{\bfseries}
\titleformat*{\subsection}{\small\bfseries}

\newcommand\independent{\protect\mathpalette{\protect\independenT}{\perp}}
\def\independenT#1#2{\mathrel{\rlap{$#1#2$}\mkern2mu{#1#2}}}

\theoremstyle{plain}
\newtheorem{theorem}{Theorem}[section]
\newtheorem{proposition}[theorem]{Proposition}
\newtheorem{lemma}[theorem]{Lemma}
\newtheorem{corollary}[theorem]{Corollary}

\theoremstyle{definition}
\newtheorem{definition}[theorem]{Definition}
\newtheorem{example}[theorem]{Example}
\newtheorem{remark}[theorem]{Remark}

\begin{document}

\maketitle

\begin{abstract}

The Sorted L-One Estimator (SLOPE) is a popular regularization method in regression, which induces clustering of the estimated coefficients. That is, the estimator can have coefficients of identical magnitude. In this paper, we derive an asymptotic distribution of SLOPE for the ordinary least squares, Huber, and Quantile loss functions, and use it to study the clustering behavior in the limit. This requires a stronger type of convergence since clustering properties do not follow merely from the classical weak convergence. For this aim, we utilize the Hausdorff distance, which provides a suitable notion of convergence for the penalty subdifferentials and a bridge toward weak convergence of the clustering pattern. We establish asymptotic control of the false discovery rate for the asymptotic orthogonal design of the regressor. We also show how to extend the framework to a broader class of regularizers other than SLOPE. 

\end{abstract}

\section{Introduction}

Consider the linear model $y = X\beta^0 + \varepsilon\;\;,$
where  $X \in  \mathbb{R}^{n\times p}$ is the design matrix, $\beta^0 \in \mathbb{R}^p$ is the vector of regression coefficients and $\varepsilon \in \mathbb{R}^n$ is the random noise vector with independent identically distributed entries $\varepsilon_1,\ldots,\varepsilon_n$.

In case when $n>p$ and the columns of $X$ are linearly independent, the vector of regression coefficients $\beta^0$ can be estimated using the classical Ordinary Least Squares (OLS) approach. According to the Gauss-Markov Theorem, OLS estimator allows to obtain the minimal mean-squared error among all linear unbiased estimators. However, it is well known that this mean-squared error can be further improved by considering regularized biased estimators, like ridge regression \cite{ridge} or LASSO \cite{tibshirani1996regression}, which usually have a substantially smaller variance than OLS. This gain in estimation accuracy can be very large when the columns of $X$ are strongly correlated and/or when $p$ approaches $n$. Moreover, regularized estimators can be applied also when $p>n$. 

Ridge regression and LASSO are designed for different data-generating regimes. While ridge regression performs the best when the vector $\beta^0$ contains many nonzero elements (i.e. it is dense), LASSO can be much better when $\beta^0$ is sparse. This gain results from the non-linear LASSO shrinkage, which allows to obtain sparse estimators of $\beta^0$. The focus of this article, Sorted L-One Estimator (SLOPE, \cite{bogdan2015slope}), is a relatively new extension of LASSO, which allows for further dimensionality reduction and improvement of estimation properties.

Given a  vector of penalties $\boldsymbol{\lambda}=(\lambda_1,\ldots,\lambda_p)$ with $\lambda_1\geq \dots \geq \lambda_p \geq 0$, the SLOPE estimator of $\beta^0$ is given by the minimization problem
\begin{equation}
\label{eq:SLOPE}
    \boldsymbol{\hat{\beta}} = \underset{\beta\in\mathbb{R}^p}{\operatorname{argmin }}  \dfrac{1}{2}\Vert y-X\beta\Vert^2_{2} + J_{\boldsymbol{\lambda}}(\beta),
\end{equation}

where $J_{\boldsymbol{\lambda}}(\beta)$ is the SLOPE norm defined as 
$$J_{\boldsymbol{\lambda}}(\beta)=\sum_{i=1}^p \lambda_i \vert \beta\vert_{(i)}\;\;,$$
and $\vert \beta\vert_{(1)}\geq \vert \beta\vert_{(2)}\geq \vert \beta\vert_{(p)}$  are the absolute values of the coordinates of $\beta$ ordered in decreasing order of their magnitude. (see Bogdan et al. \cite{bogdan2015slope}).

SLOPE has some very appealing statistical properties. In the seminal works \cite{candes, tsybakov}, it is proved that SLOPE with a specific predefined sequence of tuning parameters $\lambda$ adapts to the unknown sparsity of the vector of regression coefficients $\beta^0$ and
%, contrary to LASSO, it allows to 
obtains exact minimax estimation and prediction rates. Similar sequences of the tuning parameters yield the False Discovery Rate (FDR) control when the design matrix is orthogonal \cite{SLOPE1, bogdan2015slope}, or the asymptotic FDR control when the explanatory variables are independent \cite{Kosth,kos2020asymptotic}. The minimaxity results have been proved also for SLOPE extensions to the binary and multi-class classification \cite{felix, felix2, felix3}, identification of group of predictors \cite{grpSLOPE} or for the simultaneous  outlier detection and robust regression \cite{outliers} using the mean-shift model of \cite{meanshift1, meanshift2,ipod}. SLOPE was also investigated with the asymptotic theory of Approximate Message Passing Algorithms \cite{bu2019algorithmic,Maleki,Bu2}. While in \cite{Maleki} it is shown that in terms of estimation properties SLOPE does not systematically outperform LASSO 
%\cite{tibshirani1996regression}
or ridge regression,
%\cite{ridge},
the results in \cite{Bu2} illustrate that in terms of feature selection SLOPE has qualitatively different properties from LASSO. Specifically, SLOPE can achieve the full power of signal recovery  for the moderately sparse vectors of regression coefficients, when the power of LASSO is restricted. Also, in \cite{Bu2} it is shown that for any problem instance, there exists a choice of the tuning parameters such that SLOPE outperforms LASSO, both with respect to model selection and estimation properties. 

These superior model selection properties of SLOPE result from its ability to cluster similar values of estimated regression coefficients. Due to this property, SLOPE can efficiently exploit the situation when some parameters are equal to each other. This allows to reduce the number of estimated parameters and  enhances the speed and accuracy of model fitting. A prominent illustration of advantages resulting from such a  model reduction is provided by the Convolutional Neural Networks, where the ``parameter sharing'' has allowed to ``dramatically lower the number of unique
model parameters and to significantly increase network sizes without requiring a
corresponding increase in training data'' \cite{DLbook}. In the classical statistics setup identification of the parameter sharing patterns is very beneficial, e.g., in the context of clustering similar levels of multi-valued categorical predictors \cite{DMR, SCOPE, Pokarowski}. Also, clustering properties of SLOPE were crucial for attaining good performance in portfolio optimization and tracking of financial indices  \cite{portfolio, index}.

The ability of identification of the parameters tying patterns by SLOPE has been discussed in several articles \cite{Oscar, OWL2, index, schneider2022geometry, orthogonal}.  Specifically, in \cite{schneider2022geometry} the SLOPE clustering properties are related to the properties of the subdifferential of the SLOPE norm and the SLOPE pattern of the vector $u \in \mathbb{R}^p$ is defined by providing the characteristics which are shared by all vectors in $\mathbb{R}^p$ for which this subdifferential attains the same value as for $u$. These characteristics can be concisely described in the following definition:
\begin{definition}
A SLOPE pattern is the integer valued function $\mathbf{patt}:\mathbb{R}^p\rightarrow\mathbb{Z}^p$ given by
\begin{equation*}
     \mathbf{patt}(u)_i:=rank(\vert u_i\vert)sgn(u_i),
\end{equation*}
    where $rank(\vert u_i\vert)$ is the rank of the distinct values of $|u|$; (see \cite{orthogonal,bogdan2022pattern}). 
    %We shall also refer to an image of $\mathbf{patt}$ in $\mathbb{Z}^p$ as a pattern, and denote by $\mathfrak{P}$ the finite set of all possible patterns in $\mathbb{Z}^p$.
\end{definition}

Thus, the SLOPE pattern of the vector $u$ contains the information on the sign of its elements and the ranking of their absolute values, including the information about clusters of coordinates with the same absolute values.

In a recent article \cite{bogdan2022pattern}  a sufficient and necessary condition for the pattern recovery by SLOPE is provided.  This condition specifies the patterns which can be recovered by SLOPE for a given design matrix $X$ and when the magnitude of regression coefficients is large enough. The result was further applied to characterize the asymptotic probability of the pattern recovery in the classical asymptotic setup, when the number of features $p$ is fixed and the number of examples $n$ tends to infinity. 

In the present paper we substantially expand knowledge on the asymptotic properties of SLOPE and its ability for the pattern recovery and FDR control. Firstly, we extend the asymptotic results of Knight and Fu \cite{fu2000asymptotics} on Lasso type estimators and derive the limiting distribution of $\sqrt{n}(\boldsymbol{\hat{\beta}}_n - \beta^0)$, as a solution to a regularized random least squares problem. We also observe that, crucially, for a general regularizer, weak convergence of $\sqrt{n}(\boldsymbol{\hat{\beta}}_n - \beta^0 )$ does not guarantee weak convergence of $\mathbf{patt}(\sqrt{n}(\boldsymbol{\hat{\beta}}_n - \beta^0 ))$, as demonstrated in Example \ref{counterexample pattern convergence} in the Appendix. However, we establish that for the SLOPE estimator this is in fact so. This is achieved by utilizing the Hausdorff distance \cite{rockafellar2009variational}, which provides a suitable notion of convergence for subdifferentials and a bridge toward weak pattern convergence.

Further, building on the ideas of Pollard \cite{pollard_1985}, we investigate the limiting distribution of SLOPE, where we replace the $L_2$ norm in \eqref{eq:SLOPE} with a more general loss $\sum_{i=1}^n g(y_i, X_i,\beta)$. We formulate general conditions on $g$ under which we can derive the limiting distribution and establish that these conditions are satisfied for two important loss functions, namely the (robust) Huber loss \cite{Huber1964} and the quantile loss \cite{koenker1978regression}. With this result, we are also able to asymptotically  control the FDR for the two aforementioned losses under 
%orthogonal design
independent regressors 
by fitting the correct nuisance parameters.

\subsection{Overview}

In section 2., we give formulas for the limiting distribution of $\boldsymbol{\hat{\beta}}_n$, when $\boldsymbol{\lambda}^n\rightarrow \boldsymbol{\lambda}\geq 0$ (Theorem \ref{consistency theorem}) and for the limiting distribution of $\sqrt{n}(\boldsymbol{\hat{\beta}}_n - \beta^0 )$, when $\boldsymbol{\lambda}^n/\sqrt{n}\rightarrow \boldsymbol{\lambda}\geq 0$ (Theorem \ref{sqrt-asymptotic}). In section 3., we study the interplay between the subdifferential and pattern of SLOPE and utilize the Hausdorff distance in Proposition \ref{key subdifferential proposition} on the convergence of SLOPE penalty subdifferentials. Section 4. forms the backbone of the paper. In Theorem \ref{theorem pattern}, we prove weak pattern convergence of $\sqrt{n}(\boldsymbol{\hat{\beta}}_n - \beta^0 )$ and use this in Corollary \ref{recovery of the true signal} to provide an alternative proof for the probability of true pattern recovery (Theorem 4.1.i)\cite{bogdan2022pattern}). %In Theorem \ref{pattern attainability theorem}, we characterize patterns, which are attainable in the limit. 
Weak pattern convergence leads directly to the asymptotic FDR control under independence in Example \ref{example FDR control}. We end the section with a discussion \ref{General penalty and pattern}, on how the weak pattern convergence extends to a broader family of regularizers other than SLOPE. In section 5., we prove Theorem \ref{Pollards asymptotics with penalty}, the main result of the paper, which extends the results from Section 4 to other loss functions, including the Huber loss and the Quantile loss. %We prove the main result of the paper Theorem \ref{Pollards asymptotics with penalty}, which combines the results of Pollard \cite{pollard_1985} and Knight and Fu \cite{fu2000asymptotics} and generalizes Theorem \ref{theorem pattern}. 

\section{Asymptotic distribution of SLOPE for the standard loss }

 Consider the linear model $y = X\beta^0 + \varepsilon$, with $X=(X_1,..,X_n)^{T}$, where $X_1, X_2, \dots$ are i.i.d. centered random vectors in $\mathbb{R}^p$ with covariance matrix $C$. Further assume $\varepsilon_1, \varepsilon_2, \dots$ are i.i.d. centered random variables with variance $\sigma^2$ and $X\independent\varepsilon$.  In this section, we want to give exact formulas for the low-dimensional asymptotics with fixed $p$ and $n\rightarrow\infty$, of the SLOPE estimator for the quadratic loss:
 \begin{align}\label{SLOPE-estimator_n}
     \boldsymbol{\hat{\beta}}_n &= \underset{\beta\in\mathbb{R}^p}{\operatorname{argmin }}  \dfrac{1}{2}\Vert y-X\beta\Vert^2_{2} + J_{\boldsymbol{\lambda}^n}(\beta)\\
     &=\underset{\beta\in\mathbb{R}^p}{\operatorname{argmin }}\dfrac{1}{2}(\beta-\beta^0)^{T}X^{T}X(\beta-\beta^0)+(\beta-\beta^0)^{T}X^{T}\varepsilon+\Vert\varepsilon\Vert_{2}^2/2 + J_{\boldsymbol{\lambda}^n}(\beta)\label{expanded slope estimator}
 \end{align}
 Note that by the LLN and CLT:
 \begin{equation}\label{W_n C_n convergence}
     C_n:=\dfrac{1}{n}X^{T}X\overset{a.s.}{\longrightarrow} C\hspace{0.4cm}\text{ and }\hspace{0.4cm} W_n:=\dfrac{1}{\sqrt{n}}X^{T}\varepsilon\overset{d}{\longrightarrow} W\sim\mathcal{N}(0,\sigma^2C),
 \end{equation}
 and $\Vert
 \varepsilon\Vert^2_2/n\overset{a.s.}{\longrightarrow}\sigma^2$.  The following statements and proofs are direct extensions of the results in \cite{fu2000asymptotics}. For completeness we include the proofs.
 
\begin{theorem}\label{consistency theorem}
If $\boldsymbol{\lambda}^n/n\rightarrow\boldsymbol{\lambda}\geq0$ and $C\succ0$, then $\boldsymbol{\hat{\beta}}_n\overset{p}{\longrightarrow}\textup{argmin}_{\beta} Z(\beta)$, where
\begin{equation*}
    Z(\beta)=\dfrac{1}{2}(\beta-\beta^0)^{T}C(\beta-\beta^0)+J_{\boldsymbol{\lambda}}(\beta)
\end{equation*}
In particular, in the case $\boldsymbol{\lambda}=0$, we have consistency $\boldsymbol{\hat{\beta}}_n\overset{p}{\longrightarrow}\beta^0$.
\end{theorem}

\begin{proof}
Consider
\begin{equation*}
   Z_n(\beta):=\dfrac{1}{2n}\Vert y-X\beta\Vert^2_{2}+\dfrac{1}{n}J_{\boldsymbol{\lambda}^n}(\beta)
\end{equation*}
Note that $\boldsymbol{\hat{\beta}}_n = {\operatorname{argmin }} Z_n(\beta)$. From (\ref{expanded slope estimator}) and (\ref{W_n C_n convergence}), we get that for every $\beta\in\mathbb{R}^p$:
\begin{equation*}
    Z_n(\beta)\overset{p}{\longrightarrow} Z(\beta)+\sigma^2/2.
\end{equation*}
 The pointwise convergence together with convexity of $Z_n$'s implies uniform convergence on compact subsets; i.e., $\text{sup}_{\beta\in K}\vert Z_n(\beta)-Z(\beta)-\sigma^2/2\vert\overset{p}{\longrightarrow}0$ for every compact $K$. Moreover, positive definiteness of $C$ guarantees that the minimizer of $Z(\beta)$ is unique. It follows that $\boldsymbol{\hat{\beta}}_n = {\operatorname{argmin }} Z_n(\beta)\overset{p}{\longrightarrow}{\operatorname{argmin }}Z(\beta)$. We refer the reader to \cite{andersen1982cox} and \cite{pollard1991asymptotics} for the general convergence results. 
\end{proof}
Let $J'_{\boldsymbol{\lambda}}({\beta^0};u)$ denote the directional derivative of the SLOPE norm at $\beta^0$ in direction $u$. 
\begin{theorem}\label{sqrt-asymptotic}
If $\boldsymbol{\lambda}^n/\sqrt{n}\rightarrow\boldsymbol{\lambda}\geq0$ and $C\succ0$, then $\boldsymbol{\hat{u}}_n:= \sqrt{n}(\boldsymbol{\hat{\beta}}_n-\beta^0)\overset{d}{\longrightarrow}\boldsymbol{\hat{u}}$, where
\begin{align}
    \boldsymbol{\hat{u}}&:=\textup{argmin}_{u} V(u),\nonumber\\
    V(u) &= \dfrac{1}{2}u^{T}Cu-u^{T}W+ J'_{\boldsymbol{\lambda}}({\beta^0};u)\label{V(u)},
\end{align}
with $W\sim\mathcal{N}(0,\sigma^2C)$. Moreover, the directional derivative is given by
\begin{equation}
    J'_{\boldsymbol{\lambda}}({\beta^0};u)=\sum\limits_{i=1}^p\lambda_{\pi(i)}\left[u_i sgn(\beta^0_i)\mathbb{I}[\beta^0_i\neq0]+\vert u_i\vert\mathbb{I}[\beta^0_i=0]\right]\label{directional SLOPE derivative},
\end{equation}
%where $W\sim\mathcal{N}(0,\sigma^2C)$ and $\pi$ is a joint ordering for $\beta^0$ and $(\beta^0+u/\sqrt{n})$ for all sufficiently large $n$, i.e., $\vert (\beta^0+u/\sqrt{n})_{\pi^{-1}(1)}\vert\geq...\geq\vert (\beta^0+u/\sqrt{n})_{\pi^{-1}(p)}\vert$.

 where $\pi$ is a permutation which sorts the vector $\vert(\beta^0+u/\sqrt{n})\vert$ for all sufficiently large $n$; i.e. $\exists M>0$ s.t. $\forall n\geq M$ $ \vert (\beta^0+u/\sqrt{n})\vert_{\pi^{-1}(1)}\geq...\geq\vert (\beta^0+u/\sqrt{n})\vert_{\pi^{-1}(p)}$. 
\end{theorem}

\begin{proof}
Consider
\begin{equation}\label{V_n(u)}
     V_n(u):=\dfrac{1}{2n}u^{T}X^{T}Xu-u^{T}\dfrac{1}{\sqrt{n}} X^{T}\varepsilon+J_{\boldsymbol{\lambda}^n}(\beta^0+u/\sqrt{n})-J_{\boldsymbol{\lambda}^n}(\beta^0).
\end{equation}
By substituting $u=\sqrt{n}(\beta-\beta^0)$ it follows that $\text{argmin} V_n(u)=\sqrt{n}(\boldsymbol{\hat{\beta}}_n-\beta^0)$. The stochastic part of $V_n(u)$ converges by (\ref{W_n C_n convergence}) and Slutsky:
\begin{equation}\label{C and W notation}
\dfrac{1}{2n}u^{T}X^{T}Xu-u^{T}\dfrac{1}{\sqrt{n}} X^{T}\varepsilon\overset{d}{\longrightarrow}\dfrac{1}{2}u^{T}Cu-u^{T}W.
\end{equation}
Moreover, for a fixed $u\in\mathbb{R}^p$ there exists a permutation $\pi$, which sorts $\vert \beta^0+u/\sqrt{n}\vert$ for all sufficiently large $n$, i.e. $\exists M>0: \forall n\geq M: \vert (\beta^0+u/\sqrt{n})\vert_{\pi^{-1}(1)}\geq...\geq\vert (\beta^0+u/\sqrt{n})\vert_{\pi^{-1}(p)}$. At the same time we have $\vert \beta^0\vert_{\pi^{-1}(1)}\geq...\geq\vert \beta^0\vert_{\pi^{-1}(p)}$. Consequently
%Since $\vert \beta^0_i\vert<\vert \beta^0_j\vert$ implies $\vert \beta^0_i+u_i/\sqrt{n}\vert<\vert \beta^0_j+u_j/\sqrt{n}\vert$ for all sufficiently large $n$, $\pi$ is a joint ordering for $(\beta^0+u/\sqrt{n})$ and $\beta^0$
\begin{align}\label{J_lambda convergence}
    J_{\boldsymbol{\lambda}^n}(\beta^0+u/\sqrt{n})-J_{\boldsymbol{\lambda}^n}(\beta^0)
    &= \sum\limits_{j=1}^p\boldsymbol{\lambda}^n_j\left[\vert \beta^0+u/\sqrt{n}\vert_{\pi^{-1}(j)}-\vert \beta^0\vert_{\pi^{-1}(j)}\right]\nonumber\\
    & = \sum\limits_{i=1}^p\boldsymbol{\lambda}^n_{\pi(i)}\left[\vert \beta^0_i+u_i/\sqrt{n}\vert-\vert \beta^0_i\vert\right]\nonumber\\
    &=\sum\limits_{i=1}^p\boldsymbol{\lambda}^n_{\pi(i)}\left[(u_i/\sqrt{n})sgn(\beta^0_i) \mathbb{I}[\beta^0_i\neq0]+(\vert u_i\vert/\sqrt{n})\mathbb{I}[\beta^0_i=0]\right]\nonumber\\
    & \overset{}{\longrightarrow} \sum\limits_{i=1}^p\boldsymbol{\lambda}_{\pi(i)}\left[u_i sgn(\beta^0_i)\mathbb{I}[\beta^0_i\neq0]+\vert u_i\vert\mathbb{I}[\beta^0_i=0]\right],
\end{align}
which shows that $V_n(u)\overset{d}{\longrightarrow}V(u)$. Again, by convexity and uniqueness of the minimizer of $V(u)$, we obtain $\sqrt{n}(\boldsymbol{\hat{\beta}}_n -\beta^0)= {\operatorname{argmin }} V_n(u)\overset{d}{\longrightarrow}{\operatorname{argmin }}V(u)$.
\end{proof}

\begin{remark}
    Observe that the sum appearing in (\ref{V(u)}), is exactly the directional derivative of $J_{\boldsymbol{\lambda}}$ at the point $\beta^0$ in direction $u$, hence the formulation $J'_{\boldsymbol{\lambda}}(\beta^0,u)$. As such, $u\mapsto J'_{\boldsymbol{\lambda}}(\beta^0;u)$ is itself convex and satisfies  $\partial J'_{\boldsymbol{\lambda}}(\beta^0;u)\subset \partial J_{\boldsymbol{\lambda}}(\beta^0)$ for all $u\in\mathbb{R}^p$. Note further that the result in Theorem \ref{sqrt-asymptotic} is valid not only for the SLOPE penalty, but more generally by replacing the sequence of regularizers $J_{\boldsymbol{\lambda}^n}$ with $n^{1/2}f$ for any convex penalty $f$. Finally, we remark that we can use proximal methods to solve the optimization problem (\ref{V(u)}). The proximal operator of the directional SLOPE derivative is described in the Appendix \ref{proximal operator}. We also refer the reader to \cite{larsson2022coordinate}, where the directional derivative of the SLOPE penalty is used for a coordinate descent algorithm for SLOPE.
    \end{remark}

Furthermore, the estimator $\boldsymbol{\hat{u}}$ minimizes $V(u)$ and satisfies the optimality condition 
%Let us abbreviate the sum in (\ref{V(u)}) with $J'_{\boldsymbol{\lambda}}({\beta^0};u)$. Note that for $\beta^0=0$, $J'_{\boldsymbol{\lambda}}({\beta^0};u)$ coincides with the SLOPE norm $J_{\boldsymbol{\lambda}}(u)$. Although $J'_{\boldsymbol{\lambda}}({\beta^0};u)$ is not a norm, it is convex, since it is given by a limit of convex functions in (\ref{J_lambda convergence}). One can think of $J'_{\boldsymbol{\lambda}}({\beta^0};u)$ as a directional SLOPE derivative. Observe that the estimator $\boldsymbol{\hat{u}}$ minimizes $V(u)$ and satisfies the optimality condition 
\begin{equation}\label{main optimality condition}
    0\in Cu-W +\partial J'_{\boldsymbol{\lambda}}({\beta^0};u),
\end{equation}
which yields
%where $\partial f(u)$ denotes the subdifferential of $f:\mathbb{R}^p\rightarrow\mathbb{R}$ at $u$. 
%We shall see in the next section that for all $u\in\mathbb{R}^p:$ $\partial J'_{\boldsymbol{\lambda}}({\beta^0};u)\subset \partial J_{\boldsymbol{\lambda}}(\beta^0)$. Rewriting the optimality condition (\ref{main optimality condition}) yields
\begin{align*}
    u & \in C^{-1}(W - \partial J'_{\boldsymbol{\lambda}}({\beta^0};u))\subset C^{-1}(W - \partial J_{\boldsymbol{\lambda}}(\beta^0)).
\end{align*}
It follows that $\boldsymbol{\hat{u}} = C^{-1}(W - v)$, for some $v=v_W\in\partial J_{\boldsymbol{\lambda}}(\beta^0)$. Although the exact value of $v=v_W$ depends on $W$, we can gain some insight into the structure of $v$ merely from the fact that $v\in\partial J_{\boldsymbol{\lambda}}(\beta^0)$. As we shall see, the subdifferential $\partial J_{\boldsymbol{\lambda}}(\beta^0)$ respects the signs of $\beta^0$ and we have $sgn(v_i)=sgn(\beta^0_i)$, $\forall i$ such that $\beta^0_i\neq 0$. Consequently, sampling $C\boldsymbol{\hat{u}}$ amounts to sampling $W\sim\mathcal{N}(0,\sigma^2C)$ and then pushing each component in the opposite direction of $sgn(\beta^0_i)$. The total magnitude of the push $v$ satisfies $\sum_{i=1}^p\vert v_i\vert\leq\sum_{i=1}^p \boldsymbol{\lambda}_i$. (In fact, the total magnitude of $v$ on the non-zero clusters of $\beta^0_i$ is given by a corresponding sum of $\boldsymbol{\lambda}_i$'s ). We can think of the push by $v$ as introducing a bias, which pushes $W$ in the opposite direction of $sgn(\beta^0)$. In the special case of LASSO estimator, $\boldsymbol{\lambda}_1=\dots=\boldsymbol{\lambda}_p=c>0$ and on the non-zero entries of $\beta^0$, the bias is given by $\vert v_i \vert =c$. In case $\boldsymbol{\lambda}=0$, the subdifferential vanishes, hence $v=0$ and $\boldsymbol{\hat{u}}\sim\mathcal{N}(0,\sigma^2C^{-1})$. 

\section{Subdifferential and Pattern}
In this section we investigate a ``duality'' between the SLOPE pattern and its subdifferential. %In particular, we will make use of a representation of the SLOPE subdifferential in terms of symmetries of the pattern in (\ref{subdifferential as convex hull}).
The subdifferential of the SLOPE norm has already been explored in \cite{bu2019algorithmic,bu2020algorithmic,larsson2020strong,tardivel2020simple}. For clarity of exposition, we provide a concise derivation of the subdifferential formula from \cite{bu2019algorithmic} in terms of symmetries of the pattern (\ref{subdifferential as convex hull}). 
%Many of the ideas in this subsection have already been explored in \cite{bogdan2022pattern}, \cite{bu2019algorithmic},\cite{bu2020algorithmic},\cite{schneider2022geometry}.
In subsection \ref{Hausdorff distance and directional SLOPE derivative}, we make use of the Hausdorff distance, a suitable notion for the convergence of SLOPE subdifferentials.

%The formula (\ref{subdifferential as convex hull}) for the SLOPE subdifferential is a
Note that pattern induces an equivalence relation on $\mathbb{R}^p$; $u\sim v\iff \mathbf{patt}(u) = \mathbf{patt}(v)$. The information contained in a pattern $\mathfrak{p}=\mathbf{patt}(u) \in \mathfrak{P}$ is threefold: First, it captures the information about the partition of $\{1,..,p\}$ into clusters $\mathcal{I}(\mathfrak{p})=\mathcal{I}(u)=\{I_0,I_1,..,I_m\}$. Second, it preserves the ranking of the respective clusters, with $I_0$ corresponding to the 0-cluster and $I_m$ corresponding to the cluster of the greatest magnitude. And third, it preserves the signs of $u_i$. The full information of a pattern $\mathfrak{p}$ is captured by the ordered partition $\left(I_0, I_1^+, I_1^{-},...,I_m^{+}, I_m^{-}\right)$, where $I_j^{+}=\{i\in I_j: u_i>0\}$ and $I_j^{-}=\{i\in I_j: u_i<0\}$. 

\begin{remark} 
A crucial property of pattern is its close connection to the subdifferential, which was extensively studied in \cite{schneider2022geometry}.
In fact,
\begin{equation}\label{pattern-subdifferential}
    \mathbf{patt}(u)=\mathbf{patt}(v)\implies \partial J_{\boldsymbol{\lambda}}(u)=\partial J_{\boldsymbol{\lambda}}(v)\hspace{0,3cm} \forall u,v\in\mathbb{R}^p.
\end{equation}
Moreover, if $\lambda_1>..>\lambda_p>0$, the reverse implication also holds:
\begin{equation*}
    \mathbf{patt}(u)=\mathbf{patt}(v)\iff \partial J_{\boldsymbol{\lambda}}(u)=\partial J_{\boldsymbol{\lambda}}(v)\hspace{0,3cm} \forall u,v\in\mathbb{R}^p;
\end{equation*}
see \cite{bogdan2022pattern, schneider2022geometry}. 
\end{remark}
%\footnote{ Notice that from (\ref{pattern-subdifferential}) and (\ref{u_hat}) it follows that the law of $\boldsymbol{\hat{u}}$ is absolutely continuous w.r.t. the restricted Lebesque measure $Leb\vert_{\mathbf{patt}^{-1}(\mathfrak{p})}$.}

In particular, (\ref{pattern-subdifferential}) asserts that $\partial J_{\boldsymbol{\lambda}}(u)$ is constant on $\mathbf{patt}^{-1}(\mathfrak{p})$, which allows for a well-defined notion of a subdifferential at a pattern.
%by $\partial J_{\boldsymbol{\lambda}}(\mathfrak{p}) = \partial J_{\boldsymbol{\lambda}}(u)$ for $\forall u \in \mathbf{patt} ^{-1}(\mathfrak{p}) $. 

\begin{definition}
Let $\mathfrak{p}\in\mathfrak{P}$ be any pattern. We define the SLOPE-subdifferential at $\mathfrak{p}$ as $\partial J_{\boldsymbol{\lambda}}(\mathfrak{p}):=\partial J_{\boldsymbol{\lambda}}(u)$, where $u\in\mathbb{R}^p$ with $\mathbf{patt}(u)=\mathfrak{p}$.
\end{definition}
We introduce some helpful notation to illuminate how the subdifferential arises naturally from the symmetries of the pattern. Abbreviate the diagonal matrices with entries $+1$ or $-1$ with $\{\pm 1\}^p$ and for $\mathfrak{p}\in\mathfrak{P}$ with $\mathcal{I}(\mathfrak{p})=\{I_0,I_1,..,I_m\}$, define the sign matrix $S_{\mathfrak{p}}\in\{\pm 1\}^p$ by $(S_{\mathfrak{p}})_{ii}:=1$ for $i\in I_{0}$ and $(S_{\mathfrak{p}})_{ii}:=sgn(\mathfrak{p}_i)$ else. The chain rule yields:
\begin{align*}
\text{i})\hspace{0.3cm} & \partial  J_{\boldsymbol{\lambda}}(S\mathfrak{p})=S\partial J_{\boldsymbol{\lambda}}(\mathfrak{p})  \hspace{0.2cm} \forall S\in\{\pm 1\}^p
\\
\text{ii})\hspace{0.3cm} & \partial J_{\boldsymbol{\lambda}}(\vert\mathfrak{p}\vert)= S_{\mathfrak{p}}\partial  J_{\boldsymbol{\lambda}}(\mathfrak{p})
\end{align*}
 
\begin{comment}
Note that one could w.l.o.g. assume $\mathfrak{p}\geq 0$, and if necessary, recover the subdifferential for a general $\mathfrak{p}$ by the orthogonal transformation $S_{\mathfrak{p}}$.
\end{comment}

Further, denote by $\mathcal{S}_p$ the set of permutation matrices on $\mathbb{R}^p$ and $\mathcal{S}_I$ the subset of all $p\times p$ matrices, where the $I\times I$ submatrix is a permutation matrix and all other entries are zero. Also, consider the direct product $G=\{\pm 1\}^p\cdot\mathcal{S}_p$. In particular, let $\mathcal{S}(\mathfrak{p})$ denote the group of symmetries of $\vert\mathfrak{p}\vert$ in $G$:
\begin{align*}
    \mathcal{S}(\mathfrak{p})&:=\left\{\Sigma\in G: \Sigma\vert\mathfrak{p}\vert=\vert\mathfrak{p}\vert\right\}\\
    &=\mathcal{S}^{+/-}_{I_0}\oplus\mathcal{S}_{I_1}\oplus...\oplus\mathcal{S}_{I_m}.
\end{align*}
These are the matrices, which act as permutations on the non-zero clusters $I_1,..,I_m$; and allow additional swapping of signs on $I_0$.

Moreover, for a given $v\in\mathbb{R}^p$, consider a permutation matrix $\Pi^T_v$, which sorts $\vert v\vert$, meaning that $(\Pi^{T}_v\vert v\vert)_1\geq..\geq(\Pi^{T}_v\vert v\vert)_p$. In particular, we have
 $J_{\boldsymbol{\lambda}}(v)=\langle \boldsymbol{\lambda}, \Pi^T_v \vert v \vert \rangle =\langle \Pi_v \boldsymbol{\lambda},  \vert v \vert \rangle$.  Analogously, for any pattern $\mathfrak{p}\in\mathfrak{P}$ we define the set
\begin{equation*}
    \Pi(\mathfrak{p}):=\left\{ \Pi\in\mathcal{S}_p: (\Pi^{T}\vert\mathfrak{p}\vert)_1\geq..\geq(\Pi^{T}\vert\mathfrak{p}\vert)_p\right\},
\end{equation*}
and shall write $\Pi_{\mathfrak{p}}$ for an arbitrary, but fixed element in $\Pi(\mathfrak{p})$.

\begin{comment}
Note that $\Pi(\mathfrak{p})$ and $\vert\mathfrak{p}\vert$ have the exact same symmetries in $G$ i.e.:
\begin{equation}
    S(\mathfrak{p})=\{\Sigma\in G: \Sigma\Pi(\mathfrak{p})=\Pi(\mathfrak{p})\}.
\end{equation}
\end{comment}

\begin{example}
\def\re{\color{red}}
\def\bl{\color{blue}}
\def\gr{\color{green}}
\def\black{\color{black}}
Let $\mathfrak{p}=(\textcolor{green}{0}, \textcolor{blue}{2}, \textcolor{blue}{-2}, \textcolor{red}{1}, \textcolor{blue}{2}, \textcolor{red}{1} )^T$, $\mathcal{I}(\mathfrak{p})=\{\textcolor{green}{\{1\}},\textcolor{red}{\{4,6\}},\textcolor{blue}{\{2,3,5\}}\}$. 
\\
\[ \Sigma = 
\left(\begin{matrix}
\gr-1 & 0 & 0 & 0 & 0 & 0 \\
0 & \bl0 & \bl1 & 0 & \bl0 & 0 \\
0 & \bl0 & \bl0 & 0 & \bl1 & 0\\
0 & 0 & 0 & \re0 & 0 & \re1\\
0 & \bl1 & \bl0 & 0 & \bl0 & 0\\
0 & 0 & 0 & \re1 & 0 & \re0
\end{matrix} \right)\in\mathcal{S}(\mathfrak{p})=\textcolor{green}{\mathcal{S}^{+/-}_{I_0}}\oplus\textcolor{red}{{S}_{I_1}}\oplus\textcolor{blue}{\mathcal{S}_{I_2}}
\qquad \Pi^T_{\mathfrak{p}} = \left(\begin{matrix}
0 & \bl1 & 0 & 0 & 0 & 0 \\
0 & 0 & \bl1 & 0 & 0 & 0 \\
0 & 0 & 0 & 0 & \bl1 & 0\\
0 & 0 & 0 & \re1 & 0 & 0\\
0 & 0 & 0 & 0 & 0 & \re1\\
\gr1 & 0 & 0 & 0 & 0 & 0
\end{matrix} \right)\in\Pi(\mathfrak{p})
\]
$\Sigma\mathfrak{p}=(\textcolor{green}{0}, \textcolor{blue}{-2}, \textcolor{blue}{2}, \textcolor{red}{1}, \textcolor{blue}{2}, \textcolor{red}{1} )^T$\hspace{6cm}$\Pi^T_{\mathfrak{p}}\mathfrak{p}=( \textcolor{blue}{2}, \textcolor{blue}{-2},\textcolor{blue}{2},  \textcolor{red}{1}, \textcolor{red}{1},  \textcolor{green}{0})^T$
\end{example}
%In particular, $\Pi(\mathfrak{p})$ is invariant under the action of $\mathcal{S}(\mathfrak{p})$.
Now consider a fixed $u\in\mathbb{R}^p$, $\mathfrak{p}=\mathbf{patt}(u)$,  and $\Pi_{\mathfrak{p}}\in \Pi({\mathfrak{p})}$. We can write 
\begin{align*}
    J_{\boldsymbol{\lambda}}(u)&=\max_{\Pi\in\mathcal{S}_{p}}\left\langle \Pi\boldsymbol{\lambda}, \vert u \vert\right\rangle
    \\
    &=\left\langle \Pi_{\mathfrak{p}}\boldsymbol{\lambda},  \vert u\vert\right\rangle.
\end{align*}
\begin{comment}
Observe that $\Sigma\Pi_{\mathfrak{p}}\in\Pi(\mathfrak{p})$ for every  $\Sigma\in\mathcal{S}(\mathfrak{p})$, hence $\Pi(\mathfrak{p})$ is invariant under the action of $\mathcal{S}(\mathfrak{p})$; i.e., $\Sigma\Pi(\mathfrak{p})=\Pi(\mathfrak{p})\hspace{0,2cm}\forall{\Sigma}\in\mathcal{S}(\mathfrak{p})$. Moreover, $S(\mathfrak{p})$ is the stabilizer of $\Pi(\mathfrak{p})$ in $G$; i.e., $S(\mathfrak{p})=\{\Sigma\in G: \Sigma\Pi(\mathfrak{p})=\Pi(\mathfrak{p})\}$.
%Note that the set (which only depends on the pattern). Note that $\Pi^{-1}_{\mathfrak{p}}$ sorts $\vert v^0 \vert$.
\end{comment}
More importantly, for a sufficiently small\footnote{ We want $\varepsilon$ to be smaller than one half of the smallest gap between the clusters.} $\varepsilon>0$, we have for all $v\in B_{\varepsilon}(u)$:
\begin{align}\label{local representation}
    J_{\boldsymbol{\lambda}}(v)&=\max\left\{\left\langle \Sigma \Pi_{\mathfrak{p}}\boldsymbol{\lambda},  \vert v\vert\right\rangle: \Sigma\in\mathcal{S}(\mathfrak{p})\right\}\nonumber
    \\
    &=\max\left\{\left\langle S_{\mathfrak{p}}\Sigma \Pi_{\mathfrak{p}}\boldsymbol{\lambda},   v\right\rangle: \Sigma\in\mathcal{S}(\mathfrak{p})\right\} \overset{not.}{=:} J_{\boldsymbol{\lambda}}(\mathfrak{p},v) .
\end{align}
In general, for $f=\max\{f_1,..,f_N\}$, where $f_i$ are smooth functions, the subdifferential of $f$ at a point $x$ is given by the convex hull:
\begin{equation}\label{subdiff for the max function}
\partial f(x)=con\{\nabla f_i(x):i\in I(x)\},
\end{equation}
where $I(x):=\{i:f_i(x)=f(x)\}$; (see \cite{rockafellar2009variational} Exercise 8.31). Combining (\ref{local representation}) and (\ref{subdiff for the max function}), one obtains an explicit representation of the subdifferential of SLOPE as a convex hull:
\begin{align}\label{subdifferential as convex hull}
 \partial J_{\boldsymbol{\lambda}}(\mathfrak{p})&=con\left\{S_{\mathfrak{p}}\Sigma\Pi_{\mathfrak{p}}\boldsymbol{\lambda}:\Sigma\in\mathcal{S}(\mathfrak{p})\right\}
 \\[5pt]
 &=S_{\mathfrak{p}}\hspace{0.1cm}con\left\{\mathcal{S}(\mathfrak{p})\right\}\Pi_{\mathfrak{p}}\boldsymbol{\lambda}.\nonumber
\end{align}
Notice how the subdifferential is fully captured by the group of symmetries $\mathcal{S}(\mathfrak{p})$ of the pattern. We can also view the subdifferential as a direct sum of convex hulls on each of the clusters (see \cite{bu2019algorithmic}):
\begin{align}
    \partial J_{\boldsymbol{\lambda}}(\mathfrak{p})&=\left\{v\in\mathbb{R}^p\middle\vert 
        \begin{array}{l}
        v_{I_j}\in S_{\mathfrak{p}}con\{\mathcal{S}_{I_j}\}\Pi_{\mathfrak{p}}\boldsymbol{\lambda},\hspace{0,2cm} j=1,..,m\\[5pt]
         v_{I_0}\in S_{\mathfrak{p}}con\{\mathcal{S}^{+/-}_{I_0}\}\Pi_{\mathfrak{p}}\boldsymbol{\lambda}
         \end{array}
     \right\}\nonumber
     \\[5pt]
     &=con\{S_{\mathfrak{p}}\mathcal{S}^{+/-}_{I_0}\Pi_{\mathfrak{p}}\boldsymbol{\lambda}\}\oplus con\{S_{\mathfrak{p}}\mathcal{S}_{I_1}\Pi_{\mathfrak{p}}\boldsymbol{\lambda}\}\oplus..\oplus con\{ S_{\mathfrak{p}}\mathcal{S}_{I_m}\Pi_{\mathfrak{p}}\boldsymbol{\lambda}\}.\label{subdifferential as convex hull in clusters}
\end{align}
In other words, the subdifferential is obtained by assigning $\boldsymbol{\lambda}'s$ to each cluster according to $\Pi_{\mathfrak{p}}$ and then taking the convex hull of all possible permutations within each cluster, whilst keeping track of the signs. The convex hull $con\mathcal{S}_{I_j}$ is given by the Birkhoff polytopes on $I_j$ (also known as the doubly stochastic matrices), which are the $I_j\times I_j$ matrices, where each row and column sums to one. We refer the reader to \cite{bu2019algorithmic} for further details about the representation of the SLOPE subdifferential in terms of Birkhoff polytopes and to \cite{schneider2022geometry} and \cite{larsson2020strong} for different derivations of the SLOPE subdifferential. 

\begin{comment}
For a sufficiently small $\varepsilon$ and any $\mathbf{h}\in B_{\varepsilon}(0)$, we also have $\vert u^0+\mathbf{h} \vert= \vert v^0+\mathbf{h} \vert +\alpha_1 \mathbf{1}_{I_1}\oplus..\oplus\alpha_m\mathbf{1}_{I_m}$. Also, since $\mathcal{I}(u^0)=\mathcal{I}(v^0)$, we get for all $\mathbf{h}\in B_{\varepsilon}(0)$:
\begin{align*}
    J_{\boldsymbol{\lambda}}(u^0+\mathbf{h})&=\max\limits_{\Sigma\in\mathcal{S}_{\mathcal{I}(v^0)}}\left\langle \Sigma \Pi_{v^0}\boldsymbol{\lambda},  \vert u^0+\mathbf{h}\vert\right\rangle\\
    &=\max\limits_{\Sigma\in\mathcal{S}_{\mathcal{I}(v^0)}}\left\langle \Sigma \Pi_{v^0}\boldsymbol{\lambda},  \vert v^0+\mathbf{h}\vert\right\rangle +const.\\
    &=J_{\boldsymbol{\lambda}}(v^0+\mathbf{h}) +const.
\end{align*}
\end{comment}

\subsection{Hausdorff distance and directional SLOPE derivative}\label{Hausdorff distance and directional SLOPE derivative}
We shall now consider the directional derivative in (\ref{directional SLOPE derivative}):
\begin{equation*}
    J'_{\boldsymbol{\lambda}}({\beta^0};u)=\sum\limits_{i=1}^p\lambda_{\pi(i)}\left[u_i sgn(\beta^0_i)\mathbb{I}[\beta^0_i\neq0]+\vert u_i\vert\mathbb{I}[\beta^0_i=0]\right].
\end{equation*}

\begin{comment}
    
\begin{example}
Let $\beta^0 = (0, 0, 1, 1, -1)$, $u = (\textcolor{blue}{1}, \textcolor{blue}{-2},  \textcolor{red}{-2}, \textcolor{red}{4},  \textcolor{red}{3})$. Then for large $n$: 
\begin{equation*}
\vert 1+\textcolor{red}{4}/\sqrt{n} \vert\geq \vert 1\textcolor{red}{-2}/\sqrt{n} \vert\geq \vert (-1)+\textcolor{red}{3}/\sqrt{n} \vert \geq \vert 0\textcolor{blue}{-2}/\sqrt{n} \vert \geq \vert 0+\textcolor{blue}{1}/\sqrt{n} \vert,
\end{equation*}
%where on the non-zero cluster of $\beta^0$ the sorting follows that of $sgn(\beta^0)\odot(\textcolor{red}{-2}, \textcolor{red}{4},  \textcolor{red}{3}) = (\textcolor{red}{-2}, \textcolor{red}{4},  \textcolor{red}{-3})$ 
thus $\boldsymbol{\lambda}_{\pi(\cdot)}= (\textcolor{blue}{\lambda_{\pi(1)}}, \textcolor{blue}{\lambda_{\pi(2)}}, \textcolor{red}{\lambda_{\pi(3)}}, \textcolor{red}{\lambda_{\pi(4)}}, \textcolor{red}{\lambda_{\pi(5)}})= (\textcolor{blue}{\lambda_{5}}, \textcolor{blue}{\lambda_{4}}, \textcolor{red}{\lambda_{2}}, \textcolor{red}{\lambda_{1}}, \textcolor{red}{\lambda_{3}})$ and \begin{equation*}
    J'_{\boldsymbol{\lambda}}({\beta^0};u) = \textcolor{blue}{\lambda_{5}}\cdot\vert\textcolor{blue}{1}\vert+ \textcolor{blue}{\lambda_{4}}\cdot\vert\textcolor{blue}{-2}\vert+ \textcolor{red}{\lambda_{2}}\cdot (\textcolor{red}{-2})+ \textcolor{red}{\lambda_{1}}\cdot\textcolor{red}{4}+ \textcolor{red}{\lambda_{3}}\cdot(-1)\textcolor{red}{3}.
\end{equation*}
\end{example}
\end{comment}

Recall from (\ref{J_lambda convergence}) that for every $u$; $J_{\boldsymbol{\lambda}^n}(\beta^0+u/\sqrt{n})-J_{\boldsymbol{\lambda}^n}(\beta^0)\longrightarrow J'_{\boldsymbol{\lambda}}({\beta^0};u)$ as $\boldsymbol{\lambda}^n/\sqrt{n}\rightarrow\boldsymbol{\lambda}$. We will argue, that the convergence is carried over to the convergence of the subdifferentials 
$\partial_u(J_{\boldsymbol{\lambda}^n}(\beta^0+u/\sqrt{n}))\overset{d_H}{\longrightarrow}\partial J'_{\boldsymbol{\lambda}}({\beta^0};u)$ in the Hausdorff metric.

To that end let $B^{\delta}:=\{x\in\mathbb{R}^d: d(x,B)\leq\delta\}=B+\overline{B_{\delta}(0)}$, where $\overline{B_{\delta}(x)}$ is the closed $\delta$-ball around $x$, and denote $B^{-\delta}:=\{x\in\mathbb{R}^d: \overline{B_{\delta}(x)}\subset B\}$. For non-empty sets $A,B\subset\mathbb{R}^p$, the Hausdorff distance (also Pompeiu-Hausdorff distance) is defined as
\begin{equation*}
    d_H(A,B):=inf\{\delta\geq 0\vert A\subset B^{\delta}, B\subset A^{\delta}\};
\end{equation*}
see \cite{rockafellar2009variational}. The Hausdorff distance defines a pseudo-metric, and yields a metric on the space of all closed, bounded, non-empty subsets of $\mathbb{R}^p$. For a sequence of sets $A_n$, we write $A_n \overset{d_H}{\longrightarrow} A$, if $d_H(A_n,A)\rightarrow 0$ as $n\rightarrow\infty$. We collect several basic properties of the Hausdorff distance. First, if $A_n \overset{d_H}{\longrightarrow} A, A_n'\overset{d_H}{\longrightarrow} A'$, then also $A_n+A_n'\overset{d_H}{\longrightarrow} A+A'$. It follows that if $B_n\overset{d_H}{\longrightarrow}B$, then $B_n^{\delta}\overset{d_H}{\longrightarrow}B^{\delta}$ and $B_n^{-\delta}\overset{d_H}{\longrightarrow}B^{-\delta}$, because $d_H(B_n^{-\delta},B^{-\delta})\leq d_H(B_n,B)$. Finally, one can check that if $x_n^i\rightarrow x^i$ as $n\rightarrow\infty$ are $N$ convergent sequences of points in $\mathbb{R}^p$, then 
\begin{equation}\label{corner convergence}
    con\{x^1_n,\dots,x^N_n\}\overset{d_H}{\longrightarrow}con\{x^1,\dots,x^N\}.
\end{equation}

In particular, convergence of convex sets in Hausdorff distance is compatible with convergence in distribution of random vectors. 

\begin{comment}
First, one can check that if $A_n \overset{d_H}{\longrightarrow} A, A_n'\overset{d_H}{\longrightarrow} A'$ are two sequences converging in $\mathcal{M}$, then also
\begin{equation}
  A_n+A_n'\overset{d_H}{\longrightarrow} A+A'. \label{Hausdorff sum}
\end{equation}
Define $B^{-\delta}:=\{x\in\mathbb{R}^d: B_{\delta}(x)\subset B\}$. Let $B_n$ be a sequence of convex\footnote{Convexity is necessary: The annuli $B_n=\overline{B_1(0)}\setminus B_{1/n}(0)\overset{d_H}{\longrightarrow} \overline{B_1(0)}$, but $B^{-\delta}\not\subset B_n$. } sets in $\mathcal{M}$ such that $B_n\overset{d_H}{\longrightarrow}B$, then
\begin{equation}
    \forall\delta>0: B^{-\delta}\subset B_n\subset B^{\delta}\hspace{0,2cm} \text{eventually}.\label{Hausdorff sandwich}
\end{equation}
Indeed, $B_n\overset{d_H}{\longrightarrow}B$ implies that $B_n\subset B^{\delta}$ eventually. Similarly, for all sufficiently large $n$ we have $B\subset B_n^{\delta}$, thus $B^{-\delta}\subset (B_n^{\delta})^{-\delta}=B_n$. The last equality follows from convexity and a hyperplane separation argument.

Finally, if $x_n^i\rightarrow x^i$ as $n\rightarrow\infty$ are $N$ convergent sequences in $\mathbb{R}^p$, then 
\begin{equation}
    con\{x^1_n,\dots,x^N_n\}\overset{d_H}{\longrightarrow}con\{x^1,\dots,x^N\}.\label{corner convergence}
\end{equation}
\end{comment}

\begin{lemma}\label{Hausdorff lemma}
    Suppose $B_n\overset{d_H}{\longrightarrow}B$, where $B_n$ and $B$ are convex sets in $\mathbb{R}^p$. If $W_n\overset{d}{\longrightarrow}W$, for some $W$ with a continuous bounded density w.r.t. the Lebesque measure on $\mathbb{R}^p$. Then $\mathbb{P}[W_n\in B_n]\longrightarrow \mathbb{P}[W\in B]$.
\end{lemma}
\begin{proof}
 Let $\delta>0$ be arbitrary. Since $B_n\overset{d_H}{\longrightarrow}B$, it follows that $B_n\subset B^{\delta}$ eventually. Similarly, for all sufficiently large $n$ we have $B\subset B_n^{\delta}$, thus $B^{-\delta}\subset (B_n^{\delta})^{-\delta}= B_n$, where last equality follows from convexity\footnote{Convexity is necessary: The annuli $B_n=\overline{B_1(0)}\setminus B_{1/n}(0)\overset{d_H}{\longrightarrow} \overline{B_1(0)}$, but $\overline{B_1(0)}^{-\delta}\not\subset B_n$. } and a hyperplane separation argument. As a result, for any $\delta>0$ $ B^{-\delta}\subset B_n\subset B^{\delta}\hspace{0,2cm} \text{eventually}$. Moreover, since $B$ is convex, one can show that for every $\varepsilon>0$ there exists \footnote{In fact, the bounds with tubular sets hold uniformly over all convex sets; i.e., for each $\varepsilon>0$ there even exists a $\delta>0$ such that (\ref{B^delta approx}) holds for every convex set $B$. % Or equivalently, for each $\varepsilon>0$ there exists $\delta>0$ such that $sup_{B\in\mathcal{C}}\mathbb{P}[W\in\ B^{\delta}\setminus B^{-\delta}]<\varepsilon$, where $\mathcal{C}$ is the family of all convex sets.
} a $\delta>0$ such that 
\begin{equation}\label{B^delta approx}
    \mathbb{P}[W\in B^{\delta}\hspace{0,1cm}]-\varepsilon\leq\mathbb{P}[W\in B]\leq\mathbb{P}[W\in B^{-\delta}]+\varepsilon.
\end{equation}
 Consequently, for any $\varepsilon>0$ we can choose $\delta>0$ sufficiently small such that:
\begin{align*}
    \limsup_{n\rightarrow\infty}\mathbb{P}[W_n\in B_n]&\leq \limsup_{n\rightarrow\infty} \mathbb{P}[W_n\in B^{\delta/2}]\leq\mathbb{P}[W\in B^{\delta}]\leq\mathbb{P}[W\in B]+\varepsilon
    \\[5pt]
    \liminf_{n\rightarrow\infty}\mathbb{P}[W_n\in B_n]&\geq \liminf_{n\rightarrow\infty} \mathbb{P}[W_n\in B^{-\delta/2}]\geq\mathbb{P}[W\in B^{-\delta}]\geq\mathbb{P}[W\in B]-\varepsilon,
\end{align*}
where we have have used the Portmanteau Lemma and that $\overline{B^{\delta/2}}\subset B^{\delta}$, $(B^{-\delta/2})^{\circ}\supset B^{-\delta}$. This shows the desired convergence $\mathbb{P}[W_n\in B_n]\longrightarrow \mathbb{P}[W\in B]$.
\end{proof}

%consider $N$ points $x^1,\dots,x^N$ in $\mathbb{R}^p$ and sequences $(x^i_n)$ $\{x_1,\dots,x_N\}\subset\mathbb{R}^p$, $x_i^n$

%$\text{ If } B=con\{x_1,..,x_N\}, B_n=con\{x_1^{n},...,x_N^{n}\} \text{ such that }
 
     %\hspace{0,4cm} x_i^n\rightarrow x_i \text{ in } \mathbb{R}^p \text{ for } i\in\{1,..,N\}, \text{ then } d_H(B_n,B)\overset{}$
\begin{comment}
    \begin{align}
    \text{i})&\hspace{0,2cm} d_H(A_n, A)\rightarrow 0, d_H(A_n', A')\rightarrow 0\implies d_H(A_n+A_n', A+A')\rightarrow 0\label{Hausdorff sum}
     \\
     \text{ii})&\hspace{0,2cm} d_H(B_n, B)\rightarrow 0 \implies \forall\delta>0: B^{-\delta}\subset B_n\subset B^{\delta}\hspace{0,2cm} \text{eventually}\label{Hausdorff sandwich}
     \\
     \text{iii})&\hspace{0,2cm} \text{ If } B=con\{x_1,..,x_N\}, B_n=con\{x_1^{n},...,x_N^{n}\} \text{ such that }\label{corner convergence}
     \\ 
     &\hspace{0,4cm} x_i^n\rightarrow x_i \text{ in } \mathbb{R}^p \text{ for } i\in\{1,..,N\}, \text{ then } d_H(B_n,B)\overset{}{\longrightarrow}0\nonumber
\end{align}
\end{comment}

An immediate consequence of (\ref{subdifferential as convex hull}) and (\ref{corner convergence}) is that $\partial J_{\boldsymbol{\lambda}^n/\sqrt{n}}(\mathfrak{p})\overset{d_H}{\longrightarrow}\partial J_{\boldsymbol{\lambda}}(\mathfrak{p})\hspace{0,2cm}\forall\mathfrak{p}\in\mathfrak{P}$. To get an analogous statement for $J'_{\boldsymbol{\lambda}}(\beta^0;u)$, observe that for any fixed $u\in\mathbb{R}^p$, the pattern $\mathbf{patt}(\beta^0+u/\sqrt{n})$ eventually stabilizes\footnote{This means there exists $M>0$ such that $\forall n\geq M$ $\mathbf{patt}(\beta^0+u/\sqrt{n})=\mathbf{patt}_{\beta^0}(u)$.} as $n$ increases. This simple, but key idea of a limiting pattern lies at the core of the directional SLOPE derivative $J'_{\boldsymbol{\lambda}}({\beta^0};u)$, and deserves its own definition.
\begin{definition}
Let $u\in\mathbb{R}^p$, we define the \textit{limiting pattern of u with respect to $\beta^0$}, denoted by $\mathbf{patt}_{\beta^0}(u)$, as
\begin{equation*}
\mathbf{patt}_{\beta^0}(u):=\lim_{\varepsilon \searrow 0} \mathbf{patt}(\beta^0+\varepsilon u).
\end{equation*}
For a pattern $\mathfrak{p}\in\mathfrak{P}$ we define its limiting pattern by $\mathbf{patt}_{\beta^0}(\mathfrak{p})=\mathbf{patt}_{\beta^0}(u)$, where $u\in\mathbb{R}^p$ is any vector with $\mathbf{patt}(u)=\mathfrak{p}$. 
\end{definition}
Note that $\mathbf{patt}_{\beta^0}(\mathfrak{p})$ is well defined, since if $\mathbf{patt}(u)$=$\mathbf{patt}(u')$ then also $\mathbf{patt}_{\beta^0}(u)$=$\mathbf{patt}_{\beta^0}(u')$.%, but the reverse implication does not hold; for example if $\vert \beta^0_i\vert\neq\vert \beta^0_j\vert$ for all $i\neq j$, then $\mathbf{patt}_{\beta^0}(u)$=$\mathbf{patt}_{\beta^0}(u')$ for all $u,u'\in\mathbb{R}^p$. 

\begin{example}\label{pattern example}
An example of such a triple $\beta^0,  \mathfrak{p},  \mathbf{patt}_{\beta^0}(\mathfrak{p})$:
\[  \begin{array}{cccccccccccc}
\beta^0&=(0 & 0  & \vert3 & 3 & -3 &\vert 7 & 7 & 7 & 7 & 7&) \\
\mathfrak{p}&=(0 & \vert1 & 1 & 1 & -1 & -1 & -1 & \vert2 & -2 & -2&) \\
\mathbf{patt}_{\beta^0}(\mathfrak{p})&=(0 & \vert1 & \vert2 & 2 & -2 &\vert 4 & 4 &\vert 5 & \vert3 & 3&) 
\end{array} \] 
\end{example}

\begin{comment}
\begin{align*}
    \beta^0&=(0, 0,\vert 3, 3, -3, -3,\vert 7, 7, 7)
    \\
    \mathfrak{p}&=(0 , \vert 1, 1, 1, -1, -1, 1,\vert 2, -2)
    \\
    \mathbf{patt}_{\beta^0}(\mathfrak{p})&=(0,\vert 1,\vert 2, 2, -2, -2, \vert5, \vert4, \vert3)
\end{align*}
\end{comment}
Recall the notation in (\ref{local representation}): $J_{\boldsymbol{\lambda}}( \mathfrak{p};v)=\max\left\{\left\langle S_{\mathfrak{p}}\Sigma \Pi_{\mathfrak{p}}\boldsymbol{\lambda},   v\right\rangle: \Sigma\in\mathcal{S}(\mathfrak{p})\right\}$ and the key insight that locally around some $u\in\mathbb{R}^p$ with $\mathbf{patt}(u)=\mathfrak{p}$, we can represent $J_{\boldsymbol{\lambda}}(v)=J_{\boldsymbol{\lambda}}(\mathfrak{p},v)$. The following result asserts that the directional derivative $J'_{\boldsymbol{\lambda}}(\beta^0,v)$ has a similar local representation in a neighborhood of $u\in\mathbb{R}^p$: $J'_{\boldsymbol{\lambda}}(\beta^0,v)=J_{\boldsymbol{\lambda}}(\mathfrak{p}_0;v)$, where $\mathfrak{p}_0$ is the limiting pattern of $u$ w.r.t. $\beta^0$. 
\begin{proposition}\label{key subdifferential proposition} Given $u\in\mathbb{R}^p$, $\mathfrak{p}=\mathbf{patt}(u)$,  $\mathfrak{p}_0=\mathbf{patt}_{\beta^0}(\mathfrak{p})$, the following holds:
\begin{align}
i)\hspace{0,2cm}& \exists \varepsilon>0 \hspace{0,2cm}\forall v\in B_{\varepsilon}(u) : J'_{\boldsymbol{\lambda}}(\beta^0;v)=J_{\boldsymbol{\lambda}}(\mathfrak{p}_0;v), \\
ii)\hspace{0,2cm}&\partial J'_{\boldsymbol{\lambda}}({\beta^0}; \mathfrak{p})=\partial J_{\boldsymbol{\lambda}}(\mathfrak{p}_0),\label{subdifferential identity for generalized penalty}\\
iii)\hspace{0,2cm}&
\partial_uJ_{\boldsymbol{\lambda}^n}(\beta^0+u/\sqrt{n})\overset{d_H}{\longrightarrow}\partial J'_{\boldsymbol{\lambda}}({\beta^0}; \mathfrak{p}).\label{Hausdorff subdifferential convergence}
\end{align}
\end{proposition}

\begin{proof}
Take $M>0$ s.t. $\mathbf{patt}(\beta^0+u/\sqrt{n})=\mathfrak{p}_0$ for all $n\geq M$. By (\ref{local representation}), for all $v$ in a sufficiently small neigborhood $B_{\varepsilon}(u)$ around $u$; $J_{\boldsymbol{\lambda}^n}(\beta^0+v/\sqrt{n})=J_{\boldsymbol{\lambda}^n}(\mathfrak{p}_0;\beta^0+v/\sqrt{n} )\hspace{0,2cm} \forall n\geq M$. Moreover, observe that $J_{\boldsymbol{\lambda}^n}(\beta^0)=\left\langle \Sigma \Pi_{\beta^0}\boldsymbol{\lambda}^n,  \vert \beta^0 \vert \right\rangle=\left\langle \Sigma \Pi_{\mathfrak{p}_0}\boldsymbol{\lambda}^n,  S_{\mathfrak{p}_0} \beta^0\right\rangle$ for every $\Sigma\in\mathcal{S}(\mathfrak{p}_0)$, since $\vert \beta^0 \vert =  S_{\mathfrak{p}_0} \beta^0$, $\Pi ^T_{\mathfrak{p}_0}$ orders $\vert \beta^0 \vert $ and $\mathcal{S}(\mathfrak{p}_0) \subset \mathcal{S}(\mathbf{patt}(\beta^0))$. Hence $\forall v\in B_{\varepsilon}(u)$:
\begin{align*}
    J_{\boldsymbol{\lambda}^n}(\beta^0+v/\sqrt{n})-J_{\boldsymbol{\lambda}^n}(\beta^0)&= J_{\boldsymbol{\lambda}^n}(\mathfrak{p}_0;\beta^0+v/\sqrt{n}) -J_{\boldsymbol{\lambda}^n}(\beta^0)
    \\
    &=\max\limits_{\Sigma\in\mathcal{S}(\mathfrak{p}_0)}\left\langle S_{\mathfrak{p}_0}\Sigma \Pi_{\mathfrak{p}_0}\boldsymbol{\lambda}^n,  \beta^0+v/\sqrt{n} \right\rangle - J_{\boldsymbol{\lambda}^n}(\beta^0)
    \\
    &=\max\limits_{\Sigma\in\mathcal{S}(\mathfrak{p}_0)}\left\langle S_{\mathfrak{p}_0}\Sigma \Pi_{\mathfrak{p}_0}\boldsymbol{\lambda}^n/\sqrt{n},  v \right\rangle
    \\
    &\longrightarrow \max\limits_{\Sigma\in\mathcal{S}(\mathfrak{p}_0)}\left\langle S_{\mathfrak{p}_0}\Sigma \Pi_{\mathfrak{p}_0}\boldsymbol{\lambda},  v \right\rangle = J_{\boldsymbol{\lambda}}(\mathfrak{p}_0;v).
\end{align*}
From (\ref{J_lambda convergence}), we also know that $J_{\boldsymbol{\lambda}^n}(\beta^0+v/\sqrt{n})-J_{\boldsymbol{\lambda}^n}(\beta^0)\longrightarrow J'_{\boldsymbol{\lambda}}(\beta^0;v)$, which yields claim $i)$. As a result of $i)$; $\partial J'_{\boldsymbol{\lambda}}(\beta^0;u)=\partial J_{\boldsymbol{\lambda}}(\mathfrak{p}_0;u)=con\left\{S_{\mathfrak{p}_0}\Sigma\Pi_{\mathfrak{p}_0}\boldsymbol{\lambda}:\Sigma\in\mathcal{S}(\mathfrak{p}_0)\right\}=\partial J_{\boldsymbol{\lambda}}(\mathfrak{p}_0)$, where the second equality follows from (\ref{subdiff for the max function}) by observing that the maximum of $\left\langle S_{\mathfrak{p}_0}\Sigma \Pi_{\mathfrak{p}_0}\boldsymbol{\lambda},  u \right\rangle $ is attained for all $\Sigma\in\mathcal{S}(\mathfrak{p}_0)$. This shows $ii)$. Finally, for $n\geq M$; $\partial_uJ_{\boldsymbol{\lambda}^n}(\beta^0+u/\sqrt{n})=\partial J_{\boldsymbol{\lambda}^n/\sqrt{n}}(\mathfrak{p}_0)$, which converges in Hausdorff distance to $\partial J_{\boldsymbol{\lambda}}(\mathfrak{p}_0)$ by (\ref{subdifferential as convex hull}) and (\ref{corner convergence}), finishing the proof.
\begin{comment}
    Knowing that $J_{\boldsymbol{\lambda},\beta^0}(v)=J_{\boldsymbol{\lambda}}(v; \mathfrak{p}_0)$ holds locally around $u$, 

To show ii) take a $u^*$, with $\mathbf{patt}(u^*)=\mathfrak{p}_0$, sufficiently close to $u^0$ in i) so there exists a $B_{\varepsilon^*}(u^*) \subset B_{\varepsilon}(u^0)$. Now for all  $u\in B_{\varepsilon^*}(u^*)$ we have, by \eqref{local representation}, that $
J_{\lambda,\beta^0}(u)=J_{\lambda}(u; \mathfrak{p}_0),
$
hence ii) follows.
\end{comment}

\end{proof}

\begin{comment}
Moreover, $J'_{\boldsymbol{\lambda}}({\beta^0};u)$ is a maximum over a finite collection of smooth functions and by Lemma \ref{subdifferential convergence lemma}:
\begin{align}
    \partial_u(J_{\boldsymbol{\lambda}^n}(\beta^0+u^0/\sqrt{n}))&=con\left\{S_{\mathfrak{p}_0}\Sigma\Pi_{\mathfrak{p}_0}\boldsymbol{\lambda}^n/\sqrt{n}:\Sigma\in\mathcal{S}^{+/-}_{\mathcal{I}(\mathfrak{p}_0)}\right\}
    \\
    &\overset{d_H}{\longrightarrow} con\left\{S_{\mathfrak{p}_0}\Sigma\Pi_{\mathfrak{p}_0}\boldsymbol{\lambda}:\Sigma\in\mathcal{S}^{+/-}_{\mathcal{I}(\mathfrak{p}_0)}\right\}= \partial J_{\boldsymbol{\lambda},\beta^0}(u^0)
\end{align}
\end{comment}

\section{Weak pattern convergence}
In this section we show that on top of the weak convergence established in Theorem \ref{sqrt-asymptotic} the sequence of patterns $\mathbf{patt}(\boldsymbol{\hat{u}}_n)$ converges weakly to $\mathbf{patt}(\boldsymbol{\hat{u}})$. We remark that the weak pattern convergence does not follow merely from the weak convergence of $\boldsymbol{\hat{u}}_n$ to $\boldsymbol{\hat{u}}$, since the pattern function $\mathbf{patt}:\mathbb{R}^d\rightarrow\mathbb{Z}^d$ is not continuous. (See Example \ref{counterexample pattern convergence}.)

\begin{theorem}\label{theorem pattern} For any pattern $\mathfrak{p}\in\mathfrak{P}$ we have:
\begin{equation*}
    \mathbb{P}[\mathbf{patt}(\boldsymbol{\hat{u}}_n)=\mathfrak{p}]\xrightarrow[n\rightarrow\infty]{}\mathbb{P}[\mathbf{patt}(\boldsymbol{\hat{u}})=\mathfrak{p}].
\end{equation*}
\end{theorem}

Before proving the theorem, we make a few preparatory observations.  Using the notation from (\ref{W_n C_n convergence}),  $C_n:=n^{-1}X^{T}X$, $W_n=n^{-1/2} X^{T}\varepsilon$ and $W\sim\mathcal{N}(0,\sigma^2C)$, the optimality conditions arising from (\ref{V(u)}) and (\ref{V_n(u)}) in terms of subdifferentials read
\begin{align*}
       &  W\in C\boldsymbol{\hat{u}}+ \partial J'_{\boldsymbol{\lambda}}(\beta^0;\boldsymbol{\hat{u}}),\\
       &  W_n\in C_n \boldsymbol{\hat{u}}_n+\partial J_{\boldsymbol{\lambda}^n/\sqrt{n}}(\beta^0+\boldsymbol{\hat{u}}_n/\sqrt{n}).
\end{align*}
Given $\mathfrak{p}\in\mathfrak{P}$ and any bounded set $D^{\mathfrak{p}}\subset\mathbf{patt}^{-1}(\mathfrak{p})$, we have for sufficiently large $n$:
\begin{align}\label{u_n set general}
       \boldsymbol{\hat{u}}\in D^{\mathfrak{p}} &\iff W\in CD^{\mathfrak{p}}+ \partial J'_{\boldsymbol{\lambda}}({\beta^0}; \mathfrak{p})=:B,\\
        \boldsymbol{\hat{u}}_n\in D^{\mathfrak{p}} &\iff  W_n\in C_nD^{\mathfrak{p}}+ \partial J_{\boldsymbol{\lambda}^ n/\sqrt{n}}(\mathfrak{p}_0)=:B_n,\nonumber
\end{align}
where $\mathfrak{p}_0=\mathbf{patt}_{\beta^0}(\mathfrak{p})$, and $B_n$ is a random set. 

\begin{proof}(Theorem \ref{theorem pattern})
Let $\varepsilon>0$ be arbitrary. By tightness of $\boldsymbol{\hat{u}}_n$ there exists $M>0$ such that $\mathbb{P}[\boldsymbol{\hat{u}}_n\notin B_M(0)]<\varepsilon$ $\forall n$. For each pattern $\mathfrak{p}\in\mathfrak{P}$ define $D^{\mathfrak{p}}:=\mathbf{patt}^{-1}(\mathfrak{p})\cap B_{M}(0)$.
 Further, let $B=CD^{\mathfrak{p}}+ \partial J'_{\boldsymbol{\lambda}}({\beta^0}; \mathfrak{p})$ and
$B_n=C_nD^{\mathfrak{p}}+ \partial J_{\boldsymbol{\lambda}^ n/\sqrt{n}}(\mathfrak{p}_0)$ with $\mathfrak{p}_0=\mathbf{patt}_{\beta^0}(\mathfrak{p})$, as in (\ref{u_n set general}). We have already established in Proposition \ref{key subdifferential proposition} that $d_H(\partial J_{\boldsymbol{\lambda}^ n/\sqrt{n}}(\mathfrak{p}_0),\partial J'_{\boldsymbol{\lambda}}({\beta^0}; \mathfrak{p}))\longrightarrow 0$. Moreover, since $C_n\overset{a.s.}{\longrightarrow} C$, we have $d_H(C_n D^{\mathfrak{p}},CD^{\mathfrak{p}})\overset{a.s.}{\longrightarrow} 0$.\footnote{ Boundedness of $D^{\mathfrak{p}}$ grants continuity of the map $A\mapsto d_{H}(A D^{\mathfrak{p}}, CD^{\mathfrak{p}})$ on $M_n(\mathbb{R})$ and thus measurability of $\omega\mapsto d_{H}(C_n(\omega) D^{\mathfrak{p}}, CD^{\mathfrak{p}})$.} It follows that $d_H(B_n,B)\overset{a.s.}{\longrightarrow}0$ and Lemma \ref{Hausdorff lemma} yields
\begin{equation*}
\mathbb{P}[\boldsymbol{\hat{u}}_n\in D^{\mathfrak{p}}] = \mathbb{P}[W_n\in B_n] \xrightarrow[n\rightarrow\infty]{} \mathbb{P}[W\in B] =\mathbb{P}[\boldsymbol{\hat{u}}\in D^{\mathfrak{p}}].
\end{equation*}
Hence
\begin{align*}
\limsup_{n\rightarrow\infty}\mathbb{P}[\boldsymbol{\hat{u}}_n\in \mathbf{patt}^{-1}(\mathfrak{p})]&\leq\limsup_{n\rightarrow\infty}\mathbb{P}[\boldsymbol{\hat{u}}_n\in D^{\mathfrak{p}}]+\varepsilon=\mathbb{P}[\boldsymbol{\hat{u}}\in D^{\mathfrak{p}}]+\varepsilon,
\\
\liminf_{n\rightarrow\infty}\mathbb{P}[\boldsymbol{\hat{u}}_n\in \mathbf{patt}^{-1}(\mathfrak{p})]&\geq\liminf_{n\rightarrow\infty}\mathbb{P}[\boldsymbol{\hat{u}}_n\in D^{\mathfrak{p}}]=\mathbb{P}[\boldsymbol{\hat{u}}\in D^{\mathfrak{p}}],
\end{align*}
 and letting $\varepsilon\downarrow 0$ yields the desired pattern  convergence.
\end{proof}

\subsection{Weak pattern recovery}
As a consequence of Theorem \ref{theorem pattern}, we can characterize the asymptotic distribution of the SLOPE pattern $\mathbf{patt}(\boldsymbol{\hat{\beta}}_n)$ in terms of $\boldsymbol{\hat{u}}$.
\begin{corollary}\label{asymptotic sampling corollary}
For any pattern $\mathfrak{p}\in\mathfrak{P}$ we have:
\begin{equation*}
    \mathbb{P}[\mathbf{patt}(\boldsymbol{\hat{\beta}}_n)=\mathfrak{p}]\xrightarrow[n\rightarrow\infty]{}\mathbb{P}[\mathbf{patt}_{\beta^0}(\boldsymbol{\hat{u}})=\mathfrak{p}].
\end{equation*}
\end{corollary}
\begin{proof}
By Theorem \ref{theorem pattern}, sampling a pattern $\mathfrak{p}\in\mathfrak{P}$ according to $\mathfrak{p}\sim\mathbf{patt}(\boldsymbol{\hat{u}}_n)$ is asymptotically equivalent to sampling $\mathfrak{p}\sim\mathbf{patt}(\boldsymbol{\hat{u}})$. Therefore, sampling a pattern $\mathfrak{p}\in\mathfrak{P}$ according to  $\mathfrak{p}\sim\mathbf{patt}(\boldsymbol{\hat{\beta}}_n)=\mathbf{patt}(\beta^0+\boldsymbol{\hat{u}}_n/\sqrt{n})$ is asymptotically equivalent to sampling according to $\mathfrak{p}\sim\mathbf{patt}(\beta^0+\boldsymbol{\hat{u}}/\sqrt{n})$, which is in turn asymptotically equivalent to sampling $\mathfrak{p}\sim\mathbf{patt}_{\beta^0}(\boldsymbol{\hat{u}})$.
\end{proof}

A natural question concerns the limiting probability of pattern recovery of the true signal vector $\beta^0$. More precisely, we're interested in $\mathbb{P}[\mathbf{patt}(\boldsymbol{\hat{\beta}}_n)=\mathbf{patt}(\beta^0)]$ as $n\rightarrow\infty$.  Using properties of $\boldsymbol{\hat{u}}$, we find a new way of proving the result given in Theorem 4.2 i) \cite{bogdan2022pattern}. Let $\mathcal{I}(\beta^0)=\{I_0, I_1, \dots,I_m\}$ and write the basis vectors $\mathbf{1}_{I_1},..,\mathbf{1}_{I_{m}}$ as columns in a ``pattern'' matrix $U_{\beta^0} = S_{\beta^0}(\mathbf{1}_{I_m}\vert\dots\vert\mathbf{1}_{I_{1}})$ and let $\Lambda_0 = S_{\beta^0}\Pi_{\beta^0}\boldsymbol{\lambda}$. (In the notation in \cite{bogdan2022pattern} $U_{\beta^0}^T \Lambda_0 = \widetilde{{\Lambda}_{b_0}}$ denotes the clustered parameter.) Also, by Proposition 2.1. in \cite{bogdan2022pattern}, the dual unit ball coincides with the subdifferential at zero $\{v\in\mathbb{R}^p: J_{\boldsymbol{\lambda}}^*(v)\leq1\}=\partial J_{\boldsymbol{\lambda}}(0)$ and the subdifferential satisfies 
\begin{align}\label{subdifferential and the dual norm}
    \partial J_{\boldsymbol{\lambda}}(\beta^0) &= \left\{v\in \mathbb{R}^p: v\in \partial J_{\boldsymbol{\lambda}}(0) \text{ and } U_{\beta^0}^T (v-\Lambda_0)=0\right\}\nonumber\\
    & = \partial J_{\boldsymbol{\lambda}}(0) \cap ( \Lambda_0 + \langle U_{\beta^0} \rangle^{\perp} ),
\end{align}
where $\langle U_{\beta^0} \rangle$ denotes the column span of $U_{\beta^0}$ and $\langle U_{\beta^0} \rangle^{\perp}$ its orthogonal complement. Equivalently, we have $ \partial J_{\boldsymbol{\lambda}}(\beta^0)-\Lambda_0 = (\partial J_{\boldsymbol{\lambda}}(0) -\Lambda_0) \cap  \langle U_{\beta^0} \rangle^{\perp} $ and in particular $\partial J_{\boldsymbol{\lambda}}(\beta^0)-\Lambda_0\perp \langle U_{\beta^0} \rangle$. Finally, notice the elegant relation between the limiting pattern $\mathbf{patt}_{\beta^0}(\cdot)$ and the pattern matrix $\langle U_{\beta^0} \rangle$:
\begin{equation}\label{limiting pattern and pattern matrix}
    \{u\in\mathbb{R}^p: \mathbf{patt}_{\beta^0}(u)=\mathbf{patt}(\beta^0)\} = \langle U_{\beta^0} \rangle.
\end{equation}
\begin{corollary}[Theorem 4.2 i) in \cite{bogdan2022pattern}] \label{recovery of the true signal}
The limiting probability of pattern recovery satisfies
\begin{align*}
     &\mathbb{P}\big[\mathbf{patt}(\boldsymbol{\hat{\beta}}_n)=\mathbf{patt}(\beta^0)\big]\underset{n\rightarrow\infty}{\longrightarrow}  \mathbb{P} \big[\boldsymbol{\hat{u}} \in \langle U_{\beta^0} \rangle\big] = \mathbb{P} \big[Z\in \partial J_{\boldsymbol{\lambda}}(0) \big],\\
    &Z\sim\mathcal{N}(C^{1/2}P C^{-1/2}\Lambda_0, \sigma^2 C^{1/2}(I-P)C^{1/2}),
\end{align*}
where $P=C^{1/2}U_{\beta^0}(U_{\beta^0}^TCU_{\beta^0})^{-1}U_{\beta^0}^TC^{1/2}$ is the projection matrix onto $C^{1/2}\langle U_{\beta^0}\rangle$. 
\end{corollary}
\begin{proof}
By Corollary \ref{asymptotic sampling corollary} and (\ref{limiting pattern and pattern matrix}), we obtain:
\begin{equation*}
    \mathbb{P}[\mathbf{patt}(\boldsymbol{\hat{\beta}}_n)=\mathbf{patt}(\beta^0)] \longrightarrow \mathbb{P}[\mathbf{patt}_{\beta^0}(\boldsymbol{\hat{u}})=\mathbf{patt}(\beta^0)]=\mathbb{P} \big[\boldsymbol{\hat{u}} \in \langle U_{\beta^0} \rangle\big].
\end{equation*}
Moreover, by Proposition \ref{key subdifferential proposition} and (\ref{limiting pattern and pattern matrix}), $\forall u\in\langle U_{\beta^0} \rangle$; $\partial J'_{\boldsymbol{\lambda}}({\beta^0};u) = \partial J_{\boldsymbol{\lambda}}(\mathbf{patt}_{\beta^0}(u)) = \partial J_{\boldsymbol{\lambda}}(\beta^0) $. Consequently, the optimality condition (\ref{main optimality condition}) $W \in Cu + \partial J'_{\boldsymbol{\lambda}}({\beta^0};u)$ yields:
%$\boldsymbol{\hat{u}}=u \iff W \in Cu + \partial J'_{\boldsymbol{\lambda}}({\beta^0};u)$
\begin{align}
\boldsymbol{\hat{u}} \in \langle U_{\beta^0} \rangle &\iff W \in C\langle U_{\beta^0} \rangle + \partial J_{\boldsymbol{\lambda}}(\beta^0)\nonumber\\
        & \iff C^{-1/2}W \in C^{1/2}\langle U_{\beta^0} \rangle + C^{-1/2}\partial J_{\boldsymbol{\lambda}}(\beta^0)\nonumber\\
        &\iff \underbrace{-C^{-1/2}\Lambda_0 + C^{-1/2}W}_{=:Y} \in C^{1/2}\langle U_{\beta^0} \rangle + C^{-1/2}(\partial
        J_{\boldsymbol{\lambda}}(\beta^0)-\Lambda_0).\label{equation which needs orthogonal decomposition}
\end{align}
Furthermore, (\ref{subdifferential and the dual norm}) implies $C^{1/2}\langle U_{\beta^0} \rangle \perp C^{-1/2}(\partial J_{\boldsymbol{\lambda}}(\beta^0)-\Lambda_0)$ and because $C^{1/2}\langle U_{\beta^0}  \rangle = \langle P \rangle$, we have $C^{-1/2}(\partial J_{\boldsymbol{\lambda}}(\beta^0)-\Lambda_0)\subset \langle P\rangle^{\perp} = \langle I-P\rangle$. Decomposing $Y = PY +(I-P)Y$, (\ref{equation which needs orthogonal decomposition}) splits into two conditions: $ PY \in C^{1/2}\langle U_{\beta^0}\rangle$ and $(I-P)Y\in C^{-1/2}(\partial J_{\boldsymbol{\lambda}}(\beta^0)-\Lambda_0)$. The first condition is always satisfied and the second is equivalent to $(I-P)Y\in C^{-1/2}(\partial J_{\boldsymbol{\lambda}}(0)-\Lambda_0)$ by (\ref{subdifferential and the dual norm}). Thus (\ref{equation which needs orthogonal decomposition}) yields
\begin{align*}
 \boldsymbol{\hat{u}} \in \langle U_{\beta^0} \rangle&\iff (I-P)Y\in C^{-1/2}(\partial
J_{\boldsymbol{\lambda}}(0)-\Lambda_0)\\
&\iff \Lambda_0 + C^{1/2} (I-P)(-C^{-1/2}\Lambda_0 + C^{-1/2}W)\in \partial
J_{\boldsymbol{\lambda}}(0)\\
& \iff C^{1/2} P C^{-1/2}\Lambda_0 + C^{1/2}(I-P)C^{-1/2}W \in \partial
J_{\boldsymbol{\lambda}}(0),
\end{align*}

and using that $W\sim\mathcal{N}(0,\sigma^2 C)$, the above Gaussian vector has expectation $C^{1/2}P C^{-1/2}\Lambda_0$ and covariance matrix $\sigma^2 C^{1/2}(I-P)C^{1/2}$, which finishes the proof.
\end{proof}

\subsection{Pattern attainability}
In the following section, we want to characterize all patterns that $\boldsymbol{\hat{u}}$ can attain. This will depend on the signal vector $\beta^0$ and the penalty vector $\boldsymbol{\lambda}$.
\begin{definition}
We shall say a pattern $\mathfrak{p}\in\mathfrak{P}$ is \textit{attainable} by $\boldsymbol{\hat{u}}$, if $\mathbb{P}[\mathbf{patt}(\boldsymbol{\hat{u}})=\mathfrak{p}]>0$.
\end{definition}

First, consider a partial ordering on the partitions of $\{1,..,p\}$, where $\mathcal{I}^1\preceq\mathcal{I}^2$ if every element (cluster) in $\mathcal{I}^1$ is a subset of some element in $\mathcal{I}^2$; in that case we will say that $\mathcal{I}^1$ is a \textit{refinement} of $\mathcal{I}^2$, for example $\{\{1,2\},\{3\},\{5\},\{4,6\}\}\preceq\{\{1,2,3\},\{4,5,6\}\}$.
For $u,v\in\mathbb{R}^p$ such that $\max_i\vert u_i\vert< \min\left\{\vert v_i-v_j\vert: v_i\neq v_j \right\}/2$, we have $\mathcal{I}(v+u)\preceq\mathcal{I}(v)$, and $\mathcal{I}(v+u)\preceq\mathcal{I}(u)$. %In fact $\mathcal{I}(v+u)$is the coarsest partition, which refines both $\mathcal{I}(v)$ and $\mathcal{I}(u)$.
In particular, for $\mathfrak{p}\in\mathfrak{P}$ and $\mathfrak{p}_0=\mathbf{patt}_{\beta^0}(\mathfrak{p})$ we have both $\mathcal{I}(\mathfrak{p}_0)\preceq\mathcal{I}(\mathfrak{p})$ and $\mathcal{I}(\mathfrak{p}_0)\preceq\mathcal{I}(\beta^0)$. 
 
 In order to determine, which patterns are attainable by $\boldsymbol{\hat{u}}$, we need to consider the dimension of $\partial J_{\boldsymbol{\lambda}}(\mathfrak{p})$. To that end, using (\ref{subdifferential as convex hull in clusters}) we rewrite
 \begin{align*}
     \partial J_{\boldsymbol{\lambda}}(\mathfrak{p}) &= (\partial J_{\boldsymbol{\lambda}}(\mathfrak{p})-S_{\mathfrak{p}}\Pi_{\mathfrak{p}}\boldsymbol{\lambda})+S_{\mathfrak{p}}\Pi_{\mathfrak{p}}\boldsymbol{\lambda}
     \\
     &=con\{ S_{\mathfrak{p}}(\mathcal{S}(\mathfrak{p})-\mathbb{I})\Pi_{\mathfrak{p}}\boldsymbol{\lambda}\}+S_{\mathfrak{p}}\Pi_{\mathfrak{p}}\boldsymbol{\lambda}
     \\
    &=con\{S_{\mathfrak{p}}(\mathcal{S}^{+/-}_{I_0}-\mathbb{I}_{I_0})\Pi_{\mathfrak{p}}\boldsymbol{\lambda}\}\oplus%con\{S_{\mathfrak{p}}(\mathcal{S}_{I_1}-\mathbb{I})\Pi_{\mathfrak{p}}\boldsymbol{\lambda}\}\oplus
    ..\oplus con\{ S_{\mathfrak{p}}(\mathcal{S}_{I_m}-\mathbb{I}_{I_m})\Pi_{\mathfrak{p}}\boldsymbol{\lambda}\}+S_{\mathfrak{p}}\Pi_{\mathfrak{p}}\boldsymbol{\lambda}.
 \end{align*}
For $j\in\{1,..,m\}$, we have $dim\{con\{S_{\mathfrak{p}}(\mathcal{S}_{I_j}-\mathbb{I}_{I_j})\Pi_{\mathfrak{p}}\boldsymbol{\lambda}\}\}=\vert I_j \vert-1$, unless $\boldsymbol{\lambda}$ is constant on $I_j$. If $\boldsymbol{\lambda}$ is constant on $I_j$, the convex hull equals $0$ and is zero-dimensional. For the zero cluster $dim\{ con\{S_{\mathfrak{p}}(\mathcal{S}^{+/-}_{I_0}-\mathbb{I}_{I_0})\Pi_{\mathfrak{p}}\boldsymbol{\lambda}\}\}=\vert I_0\vert$, whenever $\boldsymbol{\lambda}$ is not constant zero on $I_0$, otherwise the convex hull is $0$. Consequently, for any $\mathfrak{p}\in\mathfrak{P}$:
\begin{align}\label{subdifferential dimension}
    dim \partial J_{\boldsymbol{\lambda}}(\mathfrak{p})&\leq  p-\vert\mathcal{I}(\mathfrak{p})\vert+1,% \hspace{0,5cm}and,
    %\\
    %dim \partial J_{\boldsymbol{\lambda}}(\mathfrak{p})&=  p-\vert\mathcal{I}(\mathfrak{p})\vert+1
\end{align} with equality if and only if $\boldsymbol{\lambda}$ is non-constant on $I_j\in \{I_1,..,I_m\}$ and not equal to zero everywhere on $I_0$. Also, consider a set of the form:
\begin{equation}\label{D set pattern}
    D^{\mathfrak{p}}=\left\{u\in\mathbf{patt}^{-1}(\mathfrak{p}): \vert u_i\vert\in K_j \text{ for } i\in I_j, j=1,..,m_\mathfrak{p} \right\},
\end{equation}
 where $K_1,..,K_{m_\mathfrak{p}}$ are ordered compact disjoint positive intervals corresponding to the non-zero clusters in $\mathcal{I}(\mathfrak{p})=\{I_0,I_1,..,I_{m_\mathfrak{p}}\}$ and $K_0=\{0\}$. We denote the family of sets of the form (\ref{D set pattern}) by $\mathcal{D}^{\mathfrak{p}}$. Using the refined partition $\{I_0, I_1^+, I_1^{-},..., I_{m_\mathfrak{p}}^{+}, I_{m_\mathfrak{p}}^{-}\}$ we can also represent any $D^{\mathfrak{p}}\in\mathcal{D}^{\mathfrak{p}}$ explicitly as
 \begin{align}
     D^{\mathfrak{p}}&=\left\{ 0_{I_0}\oplus\alpha_1(\mathbf{1}_{I_1^+}-\mathbf{1}_{I_1^-})\oplus...\oplus\alpha_m(\mathbf{1}_{I_{m_\mathfrak{p}}^+}-\mathbf{1}_{I_{m_\mathfrak{p}}^-}): \alpha_i\in K_i, i=1,..,m_\mathfrak{p}\right\}\nonumber
     \\
     &=S_{\mathfrak{p}}\left\{\alpha_1\mathbf{1}_{I_1}\oplus...\oplus\alpha_m\mathbf{1}_{I_{m}}:\alpha_i\in K_i, i=1,..,m_\mathfrak{p}\right\}\label{D set representation},
 \end{align}
 %\JW{Do we need $0_{I_0}$, I guess it might make things clearer?}
 where $\mathbf{1}_{I}$ denotes the vector in $\mathbb{R}^p$ with $1$'s on $I$ and zeros everywhere else. We can characterize all attainable patterns by $\boldsymbol{\hat{u}}$:
 
 \begin{theorem}\label{pattern attainability theorem}
 A pattern $\mathfrak{p}\in\mathfrak{P}$ with $\mathcal{I}(\mathfrak{p})=\{I_0,I_1,..,I_m\}$ is attainable by $\boldsymbol{\hat{u}}$ if and only if
  %i)\hspace{0,2cm}&\mathcal{I}(\mathfrak{p})=\mathcal{I}(\mathfrak{p}_0),\label{cluster condition}
     %\\
 \begin{align*}
     i)\hspace{0,2cm}&\mathcal{I}(\mathfrak{p})\preceq\mathcal{I}(\beta^0),%\label{refinement condition}
     \\ 
     ii) \hspace{0,2cm}& sgn(\mathfrak{p})\odot sgn(\beta^0) \text{ is constant on } I_j\in\mathcal{I}(\mathfrak{p})\hspace{0,2cm}\forall j\neq 0, %\label{sign condition}
     \\
     iii)\hspace{0,2cm}&\boldsymbol{\lambda} \text{ is not constant on } I_j\in\mathcal{I}(\mathfrak{p})\hspace{0,2cm}\forall j\neq 0,%\label{attainability condition on non-zero clusters}
     \\
     iv)\hspace{0,2cm}&\boldsymbol{\lambda} \text{ is not constant zero on } I_0\in\mathcal{I}(\mathfrak{p}). %\label{attainability condition on zero cluster}
 \end{align*}
 %where $\mathfrak{p}_0 = \mathbf{patt}_{\beta^0}(\mathfrak{p})$ is the limiting pattern of $\mathfrak{p}$ w.r.t. $\beta^0$.
\end{theorem}
 
\begin{remark}\label{attainability remark}
Condition $i)$ asserts that, with probability one, the clusters of $\boldsymbol{\hat{u}}$ are contained within the clusters of $\beta^0$. The first two conditions $i)$ and $ii)$ together are equivalent to $\mathcal{I}(\mathfrak{p})=\mathcal{I}(\mathfrak{p}_0)$, where $\mathfrak{p}_0 = \mathbf{patt}_{\beta^0}(\mathfrak{p})$ is the limiting pattern of $\mathfrak{p}$ w.r.t. $\beta^0$. Indeed, if either $i)$ or $ii)$ is violated then $\mathcal{I}(\mathfrak{p}_0)\prec\mathcal{I}(\mathfrak{p})$ is a proper refinement, see Example \ref{pattern example}.
%the condition $\mathcal{I}(\mathfrak{p})=\mathcal{I}(\mathfrak{p}_0)$ is satisfied if and only if $\mathcal{I}(\mathfrak{p})\preceq\mathcal{I}(\beta^0)$ and $S_{\beta^0}S_{\mathfrak{p}}$ has a constant sign on each of the non-zero clusters of $\mathfrak{p}$. If either of the two conditions is violated, then $\mathcal{I}(\mathfrak{p}_0)\prec\mathcal{I}(\mathfrak{p})$ is a proper refinement, see Example \ref{pattern example}. 
\end{remark}

\begin{proof}
Let $\mathfrak{p}_0 = \mathbf{patt}_{\beta^0}(\mathfrak{p})$ be the limiting pattern of $\mathfrak{p}$ w.r.t. $\beta^0$. Following the notation in (\ref{u_n set general}) and (\ref{D set pattern}), a pattern $\mathfrak{p}\in\mathfrak{P}$ is attainable by $\boldsymbol{\hat{u}}$ if and only if $0<\mathbb{P}[\boldsymbol{\hat{u}}\in D^{\mathfrak{p}}]=\mathbb{P} [W\in B]$, where $W\sim \mathcal{N}(0,\sigma^2C)$, $D^{\mathfrak{p}}\in\mathcal{D}^{\mathfrak{p}}$ and
\begin{align*}
    B&=CD^{\mathfrak{p}}+ \partial J'_{\boldsymbol{\lambda}}({\beta^0}; \mathfrak{p})\\
    &=CD^{\mathfrak{p}}+ \partial J_{\boldsymbol{\lambda}}(\mathfrak{p}_0),
\end{align*}
which is a consequence of (\ref{subdifferential identity for generalized penalty}). Therefore, a pattern $\mathfrak{p}$ is attainable if and only if $B$ is a full-dimensional subset in $\mathbb{R}^p$. From (\ref{D set representation}), we have $dim C D^{\mathfrak{p}}=dim D^{\mathfrak{p}}= \vert\mathcal{I}(\mathfrak{p})\vert-1$ and from (\ref{subdifferential dimension}) we get $dim \partial J_{\boldsymbol{\lambda}}(\mathfrak{p}_0)\leq  p-\vert\mathcal{I}(\mathfrak{p}_0)\vert+1$. It follows
\begin{align}
    dim B &\leq dim CD^{\mathfrak{p}}+ dim \partial J_{\boldsymbol{\lambda}}(\mathfrak{p}_0)\nonumber
    \\
    &\leq p -\left(\vert\mathcal{I}(\mathfrak{p}_0)\vert- \vert\mathcal{I}(\mathfrak{p})\vert\right)\label{lambda inequality}
    \\
    &\leq p.\label{refinement inequality}
\end{align}
If $i)$ or $ii)$ is violated, then by Remark \ref{attainability remark} we have  $\mathcal{I}(\mathfrak{p}_0)\prec\mathcal{I}(\mathfrak{p})$, thus the inequality in (\ref{refinement inequality}) is strict and $dim B < p$. If $iii)$ or $iv)$ are violated, the inequality in (\ref{lambda inequality}) is strict. Hence $i) - iv)$ are necessary for attainability. 

If $i) - iv)$ are satisfied, then (\ref{refinement inequality}) and (\ref{lambda inequality}) are equalities. Therefore, to see that $i) - iv)$ are sufficient for attainability, it remains to show that
\begin{comment}
    First, observe that for attainability it is necessary that $\vert\mathcal{I}(\mathfrak{p})\vert=\vert\mathcal{I}(\mathfrak{p}_0)\vert$, which is equivalent to the condition $\mathcal{I}(\mathfrak{p})=\mathcal{I}(\mathfrak{p}_0)$, because of $\mathcal{I}(\mathfrak{p}_0)\preceq\mathcal{I}(\mathfrak{p})$. 

Henceforth, we assume $\mathcal{I}(\mathfrak{p})=\mathcal{I}(\mathfrak{p}_0)$. To get equality in (\ref{lambda conditions}), we require equality in (\ref{subdifferential dimension}): $dim \partial J_{\boldsymbol{\lambda}}(\mathfrak{p}_0)\leq  p-\vert\mathcal{I}(\mathfrak{p}_0)\vert+1$, which occurs if and only if (\ref{attainability condition on non-zero clusters}) and (\ref{attainability condition on zero cluster}) are satisfied. This proves that the three conditions (\ref{cluster condition}), (\ref{attainability condition on non-zero clusters}) and (\ref{attainability condition on zero cluster}) are necessary for attainability. To see that they are also sufficient, it remains to argue that 
\end{comment}

\begin{equation*}
    dim B = dim CD^{\mathfrak{p}}+ dim \partial J_{\boldsymbol{\lambda}}(\mathfrak{p}_0).
\end{equation*}
This is a completely general fact, which is valid without any assumptions on $\mathfrak{p}, \beta^0$ or $\boldsymbol{\lambda}$. It suffices to argue that $span CD^{\mathfrak{p}}\cap span (\partial J_{\boldsymbol{\lambda}}(\mathfrak{p}_0)-S_{\mathfrak{p}_0}\Pi_{\mathfrak{p}_0}\boldsymbol{\lambda})=0$, which will guarantee linear independence of the spans and additivity of their dimensions\footnote{ The affine shift of $\partial J_{\boldsymbol{\lambda}}(\mathfrak{p}_0)$ by $S_{\mathfrak{p}_0}\Pi_{\mathfrak{p}_0}\boldsymbol{\lambda}$ does not affect the dimension.}. We set
\begin{align*}
    V(\mathfrak{p}_0)&=span (\partial J_{\boldsymbol{\lambda}}(\mathfrak{p}_0)-S_{\mathfrak{p}_0}\Pi_{\mathfrak{p}_0}\boldsymbol{\lambda})
    \\
    &=span\{ S_{\mathfrak{p}_0}(\mathcal{S}(\mathfrak{p}_0)-\mathbb{I})\Pi_{\mathfrak{p}_0}\boldsymbol{\lambda}\}.
\end{align*}
Assume that the above intersection is non-trivial, meaning that there exists a non-zero vector $\boldsymbol{\alpha}=S_{\mathfrak{p}}(\alpha_1\mathbf{1}_{I_1}\oplus...\oplus\alpha_m\mathbf{1}_{I_{m}})$ in $D^{\mathfrak{p}}$ (representation (\ref{D set representation})), such that $C\boldsymbol{\alpha}$ lies in $V(\mathfrak{p}_0)$. Then there are coefficients $\{\beta_{\Sigma}:\Sigma\in\mathcal{S}(\mathfrak{p}_0)\}$ such that
\begin{equation*}
    C\boldsymbol{\alpha}=\sum_{\Sigma\in\mathcal{S}(\mathfrak{p}_0)}\beta_{\Sigma} S_{\mathfrak{p}_0}(\Sigma-\mathbb{I})\Pi_{\mathfrak{p}_0}\boldsymbol{\lambda}.
\end{equation*} 
 Also, one can check that $S_{\mathfrak{p}_0}S_{\mathfrak{p}}$ has a constant sign on each of the non zero clusters of $\mathfrak{p}_0$. It follows that $S_{\mathfrak{p}_0}\boldsymbol{\alpha}$ is constant on each cluster of $\mathfrak{p}_0$. Therefore, $(S_{\mathfrak{p}_0}\boldsymbol{\alpha})^T(\Sigma-\mathbb{I})=0$ for any $\Sigma\in\mathcal{S}(\mathfrak{p}_0)$.  This leads to a contradiction:
\begin{equation*}
    0<\left\langle\boldsymbol{\alpha}, C\boldsymbol{\alpha}\right\rangle=
    %\sum_{j=1}^m\sum_{\Sigma\in\mathcal{S}(\mathfrak{p}_0)}\alpha_j \beta_{\Sigma} \left\langle S_{\mathfrak{p}}\mathbf{1}_{I_j},S_{\mathfrak{p}_0}(\Sigma-\mathbb{I})\Pi_{\mathfrak{p}_0}\boldsymbol{\lambda} \right\rangle=0,
    \sum_{\Sigma\in\mathcal{S}(\mathfrak{p}_0)} \beta_{\Sigma} \left\langle S_{\mathfrak{p}_0}\boldsymbol{\alpha}, (\Sigma-\mathbb{I})\Pi_{\mathfrak{p}_0}\boldsymbol{\lambda} \right\rangle=0.
\end{equation*}
\end{proof}

%As a consequence, if Consequently, as $\boldsymbol{\lambda}^n/\sqrt{n}\rightarrow\boldsymbol{\lambda}$, the corner points of $\partial J_{\boldsymbol{\lambda}^n/\sqrt{n}}(\beta^0+ u/\sqrt{n})$ converge to the corner points of 
% and $\beta^0=\beta^0_{I_{0}^{}}\oplus \beta^0_{I_{1}}\oplus...\oplus \beta^0_{I_{m}}$ the associated decomposition of $\beta^0$ into clusters of the same magnitude. For $I\subset\{1,..,p\}$ denote $f^I_{\boldsymbol{\lambda},\pi}(u_1,..,u_p):=\sum_{i\in I}\lambda_{\pi(i)} u_{i}$.

\subsection{Asymptotic FDR control}

 %We now turn our attention from the estimation problem to the question of model selection.
 We can harness the weak pattern convergence from Theorem \ref{theorem pattern} to characterize the asymptotic false discovery rate of the SLOPE estimator $\boldsymbol{\hat{\beta}}_n$ as $n$ goes to infinity. Recall the notation from Theorem \ref{sqrt-asymptotic}, where $\boldsymbol{\hat{u}}_n$ and $\boldsymbol{\hat{u}}$ are the minimizers of (\ref{V_n(u)}) and (\ref{V(u)}) respectively and $\boldsymbol{\hat{\beta}}_n=\beta^0+\boldsymbol{\hat{u}}_n/\sqrt{n}$. Moreover, under the scaling $\boldsymbol{\lambda}^n/\sqrt{n}\rightarrow\boldsymbol{\lambda}\geq0$, $\boldsymbol{\hat{u}}_n$ converges to $\boldsymbol{\hat{u}}$ both weakly and weakly in pattern. By Corollary \ref{asymptotic sampling corollary}, sampling a random pattern according to $\mathbf{patt}(\boldsymbol{\hat{\beta}}_n)$ is asymptotically equivalent to sampling a limiting pattern $\boldsymbol{\hat{\mathfrak{b}}}:=\mathbf{patt}_{\beta^0}(\boldsymbol{\hat{u}})$. 
 
 Given the true signal $\beta^0$, we formulate $p$ null-hypotheses $H_{0i}: \beta^0_i=0$; and reject $H_{0i}$ whenever $\boldsymbol{\hat{\beta}}_{n,i}\neq0$. Let $R^n=\vert\{i:\boldsymbol{\hat{\beta}}_{n,i}\neq0\}\vert$ denote the number of rejections and $V^n=\vert\{i:\boldsymbol{\hat{\beta}}_{n,i}\neq0, \beta^0_i=0\}\vert$ the number of false rejections. Similarly, let $R=\vert\{i:\boldsymbol{\hat{\mathfrak{b}}}_{i}\neq0\}\vert$ and $V=\vert\{i:\boldsymbol{\hat{\mathfrak{b}}}_{i}\neq0, \beta^0_i=0\}\vert$ be the corresponding numbers of rejections and false rejections in the limiting experiment $\boldsymbol{\hat{\mathfrak{b}}}$. Asymptotic FDR of $\boldsymbol{\hat{\beta}}_{n}$ is matched by the FDR of $\boldsymbol{\hat{\mathfrak{b}}}$: %The ultimate (and ambitious) goal is to formulate sufficient and necessary conditions on $\boldsymbol{\lambda}^n$ and $C$, such that there is asymptotic FDR control of SLOPE($\boldsymbol{\hat{\beta}}_{n}$) at level $q$:
\begin{align*}
   \text{FDR}(\boldsymbol{\hat{\beta}}_{n})
   =\mathbb{E}\left[\dfrac{V^n}{1\vee R^n}\right]
     \xrightarrow[n\rightarrow\infty]{}  \mathbb{E}\left[\dfrac{V}{1\vee R}\right]
   =\text{FDR}(\boldsymbol{\hat{\mathfrak{b}}}).
\end{align*}
Moreover, note that by consistency $\boldsymbol{\hat{\mathfrak{b}}}_i\neq 0$, whenever $\beta^0_i\neq 0$, and $\boldsymbol{\hat{\mathfrak{b}}}$ has full power.
%\begin{equation}
%    limsup_n \text{FDR}(\boldsymbol{\hat{\beta}}_{n})=limsup_n\mathbb{E}\left[\dfrac{V^n}{1\vee R^n}\right]\leq q
%\end{equation}

%Let us use the scaling $\boldsymbol{\lambda}^n/\sqrt{n}\rightarrow\boldsymbol{\lambda}\geq0$ as in Theorem (\ref{sqrt-asymptotic}) and consider the case $\beta^0=0$ such that $\boldsymbol{\hat{\beta}}_n$ and $\boldsymbol{\hat{u}}_n$ have the same support, hence FDR($\boldsymbol{\hat{\beta}}_{n}$) $=$ FDR($\boldsymbol{\hat{u}}_n$).

\begin{example}\label{example FDR control}
Suppose $C=\mathbb{I}_p$, and let $p_0=\vert I_0\vert$. The objective function (\ref{V(u)}) simplifies to 
\begin{align}
V(u)&=\dfrac{1}{2}u^{T}Cu-u^{T}W+J'_{\boldsymbol{\lambda}}({\beta^0};u)\\
&=\dfrac{1}{2} \Vert W-u \Vert_2^2-\dfrac{1}{2}\Vert W\Vert^2_2+J'_{\boldsymbol{\lambda}}({\beta^0};u),\nonumber
\end{align}
where $W\sim\mathcal{N}(0,\sigma^2\mathbb{I}_p)$. The penalty $J'_{\boldsymbol{\lambda}}({\beta^0};u)$ is separable (see Appendix \ref{proximal operator}), and it suffices to consider the minimization problem on $I_0$: $\Vert W_{I_0}-u_{I_0}\Vert^2/2+J^{I_0}_{\boldsymbol{\lambda}}(u)$, with the smallest $p_0$ lambdas appearing in the penalty. On $I_0$, the rejection procedure $\boldsymbol{\hat{\mathfrak{b}}}$ is determined by the restricted optimization and outside of $I_0$, it rejects with probability one. 

Furthermore, it is shown in \cite{bogdan2015slope} that minimization of

\begin{equation}\label{FDR orthogonal design}
    \dfrac{1}{2} \Vert \tilde{W}-\beta \Vert_2^2+J_{\boldsymbol{\lambda}}(\beta),
\end{equation}
with $\tilde{W}\sim\mathcal{N}(\beta^0,\sigma^2\mathbb{I}_p)$, and $\lambda_i$ given by the BHq coefficients:
\begin{equation*}
    \lambda_i=\sigma\lambda^{BH(q)}_i=\sigma\Phi^{-1}\left(1-\dfrac{iq}{2p}\right),
\end{equation*}
controls the FDR at level $q(p_0/p)$. The FDR control holds for any magnitudes of $\beta^0$. Letting $\vert\beta^0_i\vert\rightarrow\infty$ $\forall i\notin I_0$, the rejection procedure (\ref{FDR orthogonal design}) separates into $I_0$ and $I_0^{c}$ and is asymptotically equivalent to $\boldsymbol{\hat{\mathfrak{b}}}$, whence it follows that $\text{FDR}(\boldsymbol{\hat{\mathfrak{b}}})\leq q(p_0/p)$.
%Therefore, the limiting distribution $\boldsymbol{\hat{u}}=\textup{argmin}_{u} V(u)$ coincides with the SLOPE estimator for orthogonal design, which is known to control the FDR \cite{bogdan2015slope} at level $q$ for $\boldsymbol{\lambda}_i$ given by the BHq coefficients:
%\begin{equation*}
%    \boldsymbol{\lambda}_i=\lambda^{BH(q)}_i=\sigma^2\Phi^{-1}\left(1-\dfrac{iq}{2p}\right).
%\end{equation*}

In particular, FDR($\boldsymbol{\hat{\beta}}_{n}$) is asymptotically controlled at level $q(p_0/p)$:
\begin{equation*}
    \text{FDR}(\boldsymbol{\hat{\beta}}_{n})
     \xrightarrow[n\rightarrow\infty]{}
   \text{FDR}(\boldsymbol{\hat{\mathfrak{b}}})\leq q\dfrac{p_0}{p}.
\end{equation*}

\begin{remark}
Note that the covariance structure among the true signals does not matter for the asymptotic FDR control. To see this, consider a block diagonal covariance $C=\mathbb{I}_{I_0}\oplus \Sigma$, with identity matrix $\mathbb{I}_{I_0}$ on $I_0$ and an arbitrary positive definite $\Sigma$ on $I_0^c$. The optimization (\ref{V(u)}) separates into $I_0$ and $I_0^c$ and the rejection procedure by $\boldsymbol{\hat{\mathfrak{b}}}_{\Sigma}=\mathbf{patt}_{\beta^0}(\boldsymbol{\hat{u}})$ does not depend on $\Sigma$, as it rejects with probability one on $I_0^c$. 
Hence $\text{FDR}(\boldsymbol{\hat{\beta}}_{n})
     \longrightarrow\text{FDR}(\boldsymbol{\hat{\mathfrak{b}}}_{\Sigma})=
   \text{FDR}(\boldsymbol{\hat{\mathfrak{b}}})\leq q(p_0/p).$
%then $C=\mathbb{I}_{I_0}\oplus C_{I_0^c}$ $W_{I_0}\sim\mathcal{N}_{p_0}(0, \sigma^2\mathbb{I})$ is independent of $W_{I_0^c}\sim\mathcal{N}_{p-p_0}(0, \Sigma)$ for any non-singular covariance matrix $W = (W_{I_0},W_{I_0^c})$ with  if $\mathbb{}$
\end{remark}
\end{example}

%Invoking Theorem (\ref{sqrt-asymptotic}) that $\boldsymbol{\hat{u}}_n\overset{d}{\longrightarrow}\boldsymbol{\hat{u}}$, it might be tempting to immediately conclude that FDR($\boldsymbol{\hat{u}}_n$)$\longrightarrow$ FDR($\boldsymbol{\hat{u}}$)$\leq q$. Unfortunately, weak convergence is not enough to make this argument. 

\subsection{General penalty and pattern}\label{General penalty and pattern}
The arguments so far, have predominantly relied on the fact that the SLOPE norm is a maximum over finitely many linear functions $J_{\boldsymbol{\lambda}}(\beta)=\operatorname{max}\{\langle \Sigma\boldsymbol{\lambda},\beta \rangle: \Sigma\in \{\pm 1\}^n\cdot\mathcal{S}_p\}$. This is an example of a polyhedral gauge, which takes the form $f(x):=\max\{0, v_1^Tx,\dots,v_p^Tx\}$ for some $v_1,\dots,v_p\in\mathbb{R}^p$. Penalties given by polyhedral norms/gauges and the related notion of a model pattern is explored in \cite{tardivel:hal-03262087}. In what follows, we show how our results on weak pattern convergence extend to more general penalties, including polyhedral norms. 

To that end, let $f=\max\{f_1,..,f_N\}$ with convex and differentiable functions $f_1,\dots,f_N$. We  have seen in (\ref{subdiff for the max function}) that $\partial f(x)=con\{\nabla f_i(x):i\in I(x)\}$, where $I(x):=\{i:f_i(x)=f(x)\}$. 
This motivates the following definition of a pattern for a function $f$ at a point $x$.
\begin{definition}
    Given $f=\max\{f_1,..,f_N\}$, we define a \textit{pattern} of $f$ as an equivalence relation on $\mathbb{R}^p$ given by $x\sim_f y \iff I(x) = I(y)$, where $I(x)=\{i:f_i(x)=f(x)\}$. We also say that $I(x)\subset\{1,\dots, N\}$ is the pattern of $f$ at $x$. The set of all patterns $\mathfrak{P}_f$ is the powerset of $\{1,\dots, N\}$.
\end{definition}
 
Further, the directional derivative of $f$ at $x$ in direction $u$ satisfies
\begin{equation}\label{directional derivative equation}
    f'(x;u) = \sup\limits_{g\in\partial f(x)}\langle g,u \rangle = \max\limits_{i\in I(x)}\langle \nabla f_i(x),u\rangle,
\end{equation}
thus $\partial_u f'(x;u)=con\{\nabla f_i(x):i\in I_x(u)\}$, where $I_x(u):=\{i\in I(x): \langle \nabla f_i(x),u\rangle= f'(x;u)\}$.  We call $I_x(u)$ the \textit{limiting pattern} of $u$ with respect to $x$.

As in Theorem \ref{sqrt-asymptotic}, the minimizer 

\begin{equation*}
    \widehat{\boldsymbol{u}}_{f,n} = \operatorname{argmin}_u u^{T}C_n u/2-u^{T}W_n+ n^{1/2}[f(\beta^0+u\sqrt{n})-f(\beta^0)],
\end{equation*}
converges in distribution to 
\begin{equation*}
    \widehat{\boldsymbol{u}}_{f}= \operatorname{argmin}_u {u^{T}Cu/2-u^{T}W+ f'({\beta^0};u)}.
\end{equation*}

%$\widehat{\boldsymbol{u}}_{f}$, which minimizes ${u^{T}Cu/2-u^{T}W+ f'({\beta^0};u)}$.
It is typical for a penalty $f$ with non-linear $f_i$'s that the general pattern $I(\widehat{\boldsymbol{u}}_{f,n})$ converges weakly to $I(\widehat{\boldsymbol{u}}_{f})$, but the SLOPE pattern $\mathbf{patt}(\widehat{\boldsymbol{u}}_{f,n})$ does not converge weakly to $\mathbf{patt}(\widehat{\boldsymbol{u}}_{f})$, see Example \ref{counterexample pattern convergence} in the Appendix.

To ensure convergence of the SLOPE pattern, we assume $f=\max\{f_1,..,f_N\}$, where $f_i(x)=\langle v_i, x\rangle + g(x)$, with $v_i\in\mathbb{R}^p$ and $g$ smooth and convex. This includes all penalties of the type $J_{\boldsymbol{\lambda}}(x)+\sum_{i=1}^p \alpha_i\vert x_i\vert +\gamma_i  x_i^2$, with $\alpha_i,\gamma_i\in\mathbb{R}\hspace{0,1cm}\forall i$, which encompasses for instance the LASSO, SLOPE, elastic net, quadratic loss, and any combination thereof. %We seek conditions on $f$, such that (\ref{Hausdorff subdifferential convergence}) remains valid for.
%\textcolor{red}{This is too general, subdifferential convergence (\ref{Hausdorff subdifferential convergence}) does not hold in this generality! Need to restrict ourselves to maximizers over linear functionals. For non-linear $f_i$ s the patterns will be some smooth submanifolds in $\mathbb{R}^d$. This is problematic since then the pattern of $x+u/\sqrt{n}$ does not stabilize at a pattern which maximizes the directional derivative. Instead: Linear setup $f=max\{f_1,..f_N\}$, where $f_i(x)=v_i^Tx+g(x)$, $v_i\in\mathbb{R}^p$ and $g(x)$ convex smooth. Patterns will be linear subspaces/half-subspaces. (But $g(x)$ can probably be absorbed into the loss function in Huber section) } 
For such $f$, there is a one-to-one correspondence between patterns and sets of constant subdifferential: $\partial f(x) = \partial f(y) \iff I(x)=I(y)$, and a subdifferential at a pattern $\partial f(I)$ is well defined. Also, if $I(u)=I(v)$, then also $I_x(u)=I_x(v)$ and we can write $I_x(I)$ for the limiting pattern of $I\in\mathfrak{P}_f$ w.r.t. $x$.
%In the case of the SLOPE norm $f=J_{\boldsymbol{\lambda}}$, the definition coincides with the SLOPE pattern, if $\boldsymbol{\lambda}$ is strictly decreasing.
A direct consequence of (\ref{directional derivative equation}) is that $\partial_u f'(x; u) = \partial f (I_x(u))$, generalizing (\ref{subdifferential identity for generalized penalty}). Moreover, for a sequence of functions $f_n=\max\{f_{1,n},..,f_{N,n}\}$, under suitable conditions on $f_n$ (such as $f_n = n^{1/2}f$), we have $\partial_u f_n(x+u/\sqrt{n})\overset{d_H}{\longrightarrow}\partial_u f'(x;u)$, generalizing (\ref{Hausdorff subdifferential convergence}). Crucially, linearity assumptions on $f$ grant that $I(x+u/\sqrt{n})$ stabilizes at $I_x(u)$, which allows for the Hausdorff convergence of subdifferentials. 

\begin{theorem}
For every convex set $\mathcal{K}\subset\mathbb{R}^p$: $\mathbb{P}[\widehat{\boldsymbol{u}}_{f,n}\in\mathcal{K}]\longrightarrow\mathbb{P}[\widehat{\boldsymbol{u}}_{f}\in\mathcal{K}]$ as $n\rightarrow\infty$. In particular, $I(\widehat{\boldsymbol{u}}_{f,n})$ converges weakly to $I(\widehat{\boldsymbol{u}}_{f})$ and  $\mathbf{patt}(\widehat{\boldsymbol{u}}_{f,n})$ converges weakly to $\mathbf{patt}(\widehat{\boldsymbol{u}}_{f})$.
\end{theorem}
\begin{proof}
The sets $I^{-1}(\mathfrak{p}_f)$, where $\mathfrak{p}_f\in\mathfrak{P}_f$, are convex cones, which  partition $\mathbb{R}^p$ into sets of constant pattern. For $\varepsilon>0$ and a convex $\mathcal{K}$, there exists $M>0$ s.t. $\mathbb{P}[\widehat{\boldsymbol{u}}_{f,n}\in\mathcal{K}\setminus B_{M}(0)]<\varepsilon $ $\forall n$. For each $\mathfrak{p}_f\in\mathfrak{P}_f$ consider the bounded convex set $\mathcal{K}^{\mathfrak{p}_f}= \mathcal{K}\cap I^{-1}(\mathfrak{p}_f)\cap B_{M}(0)$. We get
%Let $\mathfrak{p}_f\in\mathfrak{P}_f$ be fixed.  For any $\varepsilon>0$ $\exists$ convex, compact, mutually disjoint sets $D^{\mathfrak{p}_f}_1,\dots,D^{\mathfrak{p}_f}_m\subset I^{-1}(\mathfrak{p}_f)$, such that $\mathbb{P}[\widehat{\boldsymbol{u}}_{f}\in I^{-1}(\mathfrak{p}_f)\setminus\cup_i D^{\mathfrak{p}_f}_i]<\varepsilon$. For a fixed convex, compact set $D^{\mathfrak{p}_f}\subset I^{-1}(\mathfrak{p}_f)$, we get 
\begin{align}
    \mathbb{P}\left[\widehat{\boldsymbol{u}}_{f,n}\in \mathcal{K}^{\mathfrak{p}_f}\right] = & \mathbb{P}\left[W_n\in C_n \mathcal{K}^{\mathfrak{p}_f}+\partial f\left(\beta^0 +\mathcal{K}^{\mathfrak{p}_f}/\sqrt{n}\right)\right]\\
    \rightarrow 
    & \mathbb{P}\left[W\in C\mathcal{K}^{\mathfrak{p}_f}+\partial f'(\beta^0;\mathfrak{p}_f)\right]=\mathbb{P}\left[\widehat{\boldsymbol{u}}_{f}\in \mathcal{K}^{\mathfrak{p}_f}\right],
\end{align}
where the convergence follows by Lemma \ref{Hausdorff lemma}, since $I(\beta^0 +u/\sqrt{n})$ stabilizes for every $u\in\mathcal{K}^{\mathfrak{p}_f}$ at the limiting pattern $I_{\beta^0}(\mathfrak{p}_f)$ and $\partial f'(\beta^0;\mathfrak{p}_f)=\partial f (I_{\beta^0}(\mathfrak{p}_f))$.
\begin{align*}
    \limsup_{n\rightarrow\infty}\mathbb{P}\left[\widehat{\boldsymbol{u}}_{f,n}\in \mathcal{K}\right]
    &\leq \limsup_{n\rightarrow\infty} \sum_{\mathfrak{p}_f\in\mathfrak{P}_f}\mathbb{P}\left[\widehat{\boldsymbol{u}}_{f,n}\in \mathcal{K}^{\mathfrak{p}_f}\right]+\varepsilon
    = \sum_{\mathfrak{p}_f\in\mathfrak{P}_f}\mathbb{P}\left[\widehat{\boldsymbol{u}}_{f}\in \mathcal{K}^{\mathfrak{p}_f}\right]+\varepsilon
    \leq\mathbb{P}\left[\widehat{\boldsymbol{u}}_{f}\in \mathcal{K}\right]+\varepsilon,
    \\
    \liminf_{n\rightarrow\infty}\mathbb{P}\left[\widehat{\boldsymbol{u}}_{f,n}\in \mathcal{K}\right]
    &\geq \liminf_{n\rightarrow\infty} \sum_{\mathfrak{p}_f\in\mathfrak{P}_f}\mathbb{P}\left[\widehat{\boldsymbol{u}}_{f,n}\in \mathcal{K}^{\mathfrak{p}_f}\right]
    = \sum_{\mathfrak{p}_f\in\mathfrak{P}_f}\mathbb{P}\left[\widehat{\boldsymbol{u}}_{f}\in \mathcal{K}^{\mathfrak{p}_f}\right]
    \geq\mathbb{P}\left[\widehat{\boldsymbol{u}}_{f}\in \mathcal{K}\right],
\end{align*}
which shows that $\widehat{\boldsymbol{u}}_{f,n}$ converges to $\widehat{\boldsymbol{u}}_{f}$ on all convex sets. The remaining assertions of the Theorem on weak pattern convergence follow from setting $\mathcal{K}=I^{-1}(\mathfrak{p}_f)$ for some $\mathfrak{p}_f\in\mathfrak{P}_f$ or $\mathcal{K}=\mathbf{patt}^{-1}(\mathfrak{p})$ for $\mathfrak{p}\in\mathfrak{P}$, respectively.
\end{proof}

%If $f_1,..f_N$ are linear, , as in the case of SLOPE. If $f_1,..,f_N$ are $C^{1}$, (\ref{corner convergence}) implies that for any $I\in\mathfrak{P}_f$ the map $x\mapsto \partial f(x)$ from $\{x\in\mathbb{R}^p :I(x)=I\}$ to $(\mathcal{M},d_H)$ is continuous. 
%such that $\nabla f_{i,n}(x) \longrightarrow \nabla f_{i}(x)$, then by (\ref{corner convergence}) we immediately get $\partial_u f_n'(x; u) = \partial f_n(I_x(u)) \overset{d_H}{\longrightarrow}\partial f(I_x(u))$, generalizing (\ref{Hausdorff subdifferential convergence}). 
In this framework, %the general notion of a pattern extends 
the results on weak SLOPE pattern convergence in Section 4. and 5. extend to a broader family of regularizers $f$. We decide not to pursue this line of inquiry in this paper, and formulate and prove the subsequent results in the context of SLOPE.

\section{General Loss function}
Recall the assumption of the linear model $y = X\beta^0 + \varepsilon$, where $X=(X_1,..,X_n)^{T}$,  $X_i\sim_{i.i.d}\mathbb{P}$, with covariance $C$, $\varepsilon_1,..,\varepsilon_n$ i.i.d. and $X\independent\varepsilon$. We can rewrite (\ref{SLOPE-estimator_n}) as
\begin{equation}
     \boldsymbol{\hat{\beta}}_n = \underset{b\in\mathbb{R}^p}{\operatorname{argmin}}  \dfrac{1}{n}\sum_{i=1}^{n} g(y_i, X_i, \beta) + J_{\boldsymbol{\lambda}^n/n}(\beta),
\end{equation}
where we have sofar assumed $g(y_i, X_i, \beta)=( y_i- X_i^T\beta)^2/2=(\varepsilon_i - X^T_i(\beta-\beta^0))^2/2$. In this section, we show how the arguments from previous sections can be extended to a broader class of loss functions, including the Quantile and the Huber loss:
\begin{equation*} g(y_i, X_i, \beta) =
    \begin{cases}
      ( y_i- X_i^T\beta)^2/2 &  \text{Quadratic loss, } \\
      H(y_i- X_i^T\beta) & \text{Huber loss,} \\
      \vert y_i- X_i^T\beta\vert_{\alpha} & \text{Quantile loss,}
    \end{cases}
\end{equation*}
\begin{comment}
\begin{equation*} g(y_i, X_i, \beta) =
    \begin{cases}
      \vert\varepsilon_i - X^T_i(\beta-\beta^0)\vert^2 &  \text{Quadratic loss, } \\
      H(\varepsilon_i - X^T_i(\beta-\beta^0)) & \text{Huber loss,} \\
      \vert\varepsilon_i -  X^T_i(\beta-\beta^0)\vert_{\alpha} & \text{Quantile loss,}
    \end{cases}
\end{equation*}
\end{comment}
where $\vert \cdot \vert_{\alpha} :\mathbb{R}\rightarrow\mathbb{R}$, given by
\begin{align*}
    \vert x \vert_{\alpha}&:=\begin{cases}
      (1-\alpha)(-x) &  x\leq0, \\
      \alpha x &  x>0,
    \end{cases}
\end{align*}
and $\alpha\in (0,1).$ Note that $\vert x \vert_{1/2}=\frac{1}{2}\vert x\vert$. For the Quadratic loss and the Huber loss, we shall assume that $\varepsilon_i$ is centered with variance $\sigma^2$. For the Quantile loss, we specify the assumptions in the example section.
We apply the results from Pollard's paper \cite{pollard_1985}, harnessing the notion of stochastic differentiability. We consider 
\begin{align}\label{G_n}
    G_n(\beta)&:=\dfrac{1}{n}\sum_{i=1}^n g(y_i, X_i, \beta),\\
    G(\beta)&:=\mathbb{E}[g(\cdot,\beta)],\nonumber\\
    \nu_n g(\cdot,\beta)&:= n^{-1/2}\sum_{i=1}^n (g(\cdot,\beta)-G(\beta)).\nonumber
\end{align}
Observe that $G(\beta)=G_n(\beta)+n^{-1/2}\nu_n g(\cdot,\beta)$.
Throughout, we shall express $g(\cdot, \beta)$ in the form
\begin{equation}\label{key expansion of g}
    g(\cdot,\beta)=g(\cdot,\beta^0)+(\beta-\beta^0)^T\triangle(\cdot, \beta^0)+ \Vert \beta-\beta^0\Vert r(\cdot, \beta),
\end{equation}
where $\triangle(\cdot, \beta^0)$ is a random vector, which does not depend on $\beta$, and $r(\cdot, \beta)$ a rest term. Pollard \cite{pollard_1985} shows how the rather weak assumptions of twice differentiability of $G(\beta)$ together with a stochastic differentiability condition on $g(\cdot,\beta)$ are sufficient for asymptotic normality of the minimizer in (\ref{G_n}), centered around $\beta^0$. 

These conditions lend themselves well to describe the asymptotics of the SLOPE estimator $\boldsymbol{\hat{\beta}}_n$, which minimizes $G_n(\beta)+J_{\boldsymbol{\lambda}^n/n}(\beta)$. We refer the reader to \cite{pollard_1985} for a more thourough discussion, but include the definition of stochastic differentiability.
\begin{definition}
A function $g(\cdot,\beta)$ of the form (\ref{key expansion of g}) is called stochastically differentiable at $\beta^0$ if for any shrinking sequence of neighborhoods $U_n$ around $\beta^0$: 
\begin{equation*}
    \underset{\beta\in U_n}{\operatorname{sup}}\hspace{0,1cm} \dfrac{\vert\nu_n r(\cdot,\beta)\vert}{1+n^{1/2}\Vert\beta-\beta^0\Vert}\xrightarrow[n\rightarrow\infty]{p} 0.
\end{equation*}
\end{definition}
\begin{remark}
     Stochastic differentiability is a relatively weak condition on $g(\cdot,\beta)$. If $g(\cdot, \beta)$ is differentiable at $\beta^0$, and  $\triangle(\cdot,\beta^0)=\nabla_{\beta}g(\cdot,\beta^0)$ has a finite second moment $\mathbb{E}[\triangle(\cdot,\beta^0)\triangle(\cdot,\beta^0)^T]<\infty$, then $G(\beta)$ has a non-singular Hessian at $\beta^0$ and $g(\cdot,\beta)$ is stochastically differentiable at $\beta^0$ (Section 4. in \cite{pollard_1985}). However, differentiability of $g(\cdot,\beta)$ is not necessary for stochastic differentiability. More general ways of verifying stochastic differentiability, including the Huber loss or Median loss, are discussed in Section 5. \cite{pollard_1985}.
\end{remark}
The following Lemma is a direct corollary the proof of Theorem 2 in \cite{pollard_1985}. 
\begin{lemma}\label{Jonas Authority lemma} Assume $G(\beta)$ has a non-singular Hessian $\tilde{C}=\nabla^2\vert_{\beta=\beta^0}G(\beta)$ at it's minimizing value $\beta^0$ and $g(\cdot,\beta)$ is stochastically differentiable.  Then
\begin{equation*}
    n\left(G_n\left(\beta^0+\dfrac{u}{\sqrt{n}}\right)-G_n(\beta^0)\right)=\dfrac{1}{2}u^T\tilde{C}u+u^T\left[\nu_n\triangle(\cdot, \beta^0)+R_n(u)\right],
\end{equation*}
where $sup_{u\in K}\Vert R_n(u)\Vert\rightarrow 0$ in probability as $n\rightarrow\infty$, for any compact set $K$.
\end{lemma}
\begin{proof} Abbreviate $\beta_n=\beta^0+u/\sqrt{n}$. We have 
\begin{equation*}
    G_n\left(\beta_n\right)-G_n(\beta^0)= G\left(\beta_n\right)-G(\beta^0)+n^{-1/2}\nu_n\left[g\left(\cdot,\beta_n\right)-g(\cdot,\beta^0)\right].
\end{equation*}
Taylor expansion of $G(\beta)$ around $\beta^0$ yields:
 \begin{align*}
     n(G(\beta_n)-G(\beta^0))&=n[(\beta_n-\beta^0)^T \tilde{C}(\beta_n-\beta^0)/2+ o(\Vert\beta_n-\beta^0\Vert^2)]\\
     &=u^T\tilde{C}u/2 + n\cdot o(\Vert u\Vert^2/n).
 \end{align*}
 At the same time
 \begin{align*}
     n\cdot n^{-1/2}\nu_n[g(\cdot,\beta_n)-g(\cdot,\beta^0)]&=n^{1/2}(\beta_n-\beta^0)^T \nu_n\triangle(\cdot,\beta^0)+n^{1/2}\Vert(\beta_n-\beta^0)\Vert \nu_n r(\cdot,\beta_n)\\
     &=u^T \nu_n\triangle(\cdot,\beta^0)+\Vert u\Vert \nu_n r(\cdot,\beta^0+ u/\sqrt{n}).
 \end{align*}
 Combining both expressions yields the desired representation with a rest term 
 \begin{equation*}
R_n(u)= u\cdot \nu_n r(\cdot,\beta^0+ u/\sqrt{n}) + u\cdot \dfrac{o(\Vert u\Vert^2/n)}{(\Vert u\Vert^2/n)}.
 \end{equation*}
 Stochastic differentiability guarantees that for any compact set $K$, $sup_{u\in K}\Vert R_n(u)\Vert\rightarrow 0$ in probability.
\end{proof}

The following statement unifies Theorem 2 in \cite{pollard_1985} and Theorem 2 in \cite{fu2000asymptotics} into one result.
\begin{theorem}\label{Pollards asymptotics with penalty}
Let $g(\cdot,\beta)$ be a convex loss function and $\boldsymbol{\lambda}^n/\sqrt{n}\rightarrow \boldsymbol{\lambda}\geq0$. Suppose that
\begin{enumerate}
    \item $G(\beta)$ has a non-singular Hessian $\tilde{C}=\nabla^2\vert_{\beta=\beta^0}G(\beta)$ at its minimizing value $\beta^0$,
    %\item the minimizer $\beta_n$ of $G_n(\beta)$ converges in probability to $\beta^0$,
    \item $\triangle(\cdot,\beta^0)$ in (\ref{key expansion of g}) satisfies $\mathbb{E}(\triangle(\cdot,\beta^0))=0$ and $C_{\triangle}:=\mathbb{E}(\triangle(\cdot,\beta^0)\triangle(\cdot,\beta^0)^T)<\infty$,
    \item $g(\cdot, \beta)$ is stochastically differentiable at $\beta^0$.
\end{enumerate}

Then $\boldsymbol{\hat{u}}_n=\sqrt{n}(\boldsymbol{\hat{\beta}}_n-\beta^0)$ converges in distribution to the random vector $\boldsymbol{\hat{u}}:=\textup{argmin}_{u} V(u)$, where
\begin{equation}
    V(u) = \dfrac{1}{2}u^{T}\tilde{C}u-u^{T}\tilde{W}+J'_{\boldsymbol{\lambda}}({\beta^0};u),\label{V(u) general}
\end{equation}
and $\tilde{W}\sim\mathcal{N}(0,C_{\triangle})$. Moreover, $\mathbf{patt}(\boldsymbol{\hat{u}}_n)$ converges in distribution to $\mathbf{patt}(\boldsymbol{\hat{u}})$, hence $\mathbf{patt}(\boldsymbol{\hat{\beta}}_n)$ converges in distribution to $\mathbf{patt}_{\beta^0}(\boldsymbol{\hat{u}})$ and $FDR(\boldsymbol{\hat{\beta}}_n)\rightarrow FDR(\mathbf{patt}_{\beta^0}(\boldsymbol{\hat{u}}))$.
\end{theorem}
\begin{remark}
In the case $\boldsymbol{\lambda}=0$, the minimum is attained at $u=\tilde{C}^{-1}\tilde{W}\sim\mathcal{N}(0, \tilde{C}^{-1}C_{\triangle}\tilde{C}^{-1})$, which is exactly the assertion of Theorem 2 in \cite{pollard_1985}. %The proof carefully follows the proof of Lemma 3 and Theorem 2 in \cite{pollard_1985}.
\end{remark}

\begin{proof}
\begin{comment}
    First, we argue that $\boldsymbol{\hat{\beta}}_n$ is a consistent estimator of $\beta^0$. Indeed, the convex objective $G_n(\beta)+J_{\boldsymbol{\lambda}^n/n}(\beta)$ converges pointwise to $G(\beta)$, which by $i)$ has a unique minimizer at $\beta^0$, and thus $\boldsymbol{\hat{\beta}}_n\rightarrow \beta^0$ in probability (see Corollary II.2 in \cite{andersen1982cox}). 
By an analogous argument as in Lemma 3 in \cite{pollard_1985}, consistency plus the assumptions $(i)-(iii)$ guarantee $\Vert\boldsymbol{\hat{\beta}}_n-\beta^0\Vert=O_p(n^{-1/2})$. Consequently, by stochastic differentiability; $\nu_n r(\cdot,\boldsymbol{\hat{\beta}}_n)=o_p(1+n^{1/2}\Vert\boldsymbol{\hat{\beta}}_n-\beta^0\Vert)=o_p(1)$.
\end{comment} 
By Lemma \ref{Jonas Authority lemma}:  
%We need to show that also $\Vert\boldsymbol{\hat{\beta}}_n-\beta^0\Vert=O_p(n^{-1/2})$. By triangle inequality, it suffices to argue that $\Vert\boldsymbol{\hat{\beta}}_n-\beta_n\Vert=O_p(n^{-1/2})$. Now consider the process $\sqrt{n}(\boldsymbol{\hat{\beta}}_n-\beta_n)$.
\begin{align}
    \sqrt{n}(\boldsymbol{\hat{\beta}}_n-\beta^0)
    &=
    \underset{u}{\operatorname{argmin}}\hspace{0,1cm} G_n\left(\beta^0+\dfrac{u}{\sqrt{n}}\right)+J_{\boldsymbol{\lambda}^n/n}\left(\beta^0+\dfrac{u}{\sqrt{n}}\right)\nonumber\\
    &=
   \underset{u}{\operatorname{argmin}}\hspace{0,1cm}  n\left(G_n\left(\beta^0+\dfrac{u}{\sqrt{n}}\right)-G_{n}(\beta^0)\right) + J_{\boldsymbol{\lambda}^n}\left(\beta^0+\dfrac{u}{\sqrt{n}}\right)-J_{\boldsymbol{\lambda}^n}\left(\beta^0\right)\nonumber\\
    &=\underset{u}{\operatorname{argmin}}\hspace{0,1cm} \dfrac{1}{2}u^T\tilde{C}u+u^T[\nu_n\triangle(\cdot,\beta^0)+R_n(u)]+J_{\boldsymbol{\lambda}^n}\left(\beta^0+\dfrac{u}{\sqrt{n}}\right)-J_{\boldsymbol{\lambda}^n}\left(\beta^0\right),\label{ultimate minimization problem}
\end{align}
For every fixed $u$, the convex objective in (\ref{ultimate minimization problem}) converges in distribution and it follows that
\begin{align*}
    \sqrt{n}(\boldsymbol{\hat{\beta}}_n-\beta^0)\overset{d}{\longrightarrow}\underset{u}{\operatorname{argmin}} \left\{\hspace{0,1cm}  \dfrac{1}{2}u^T\tilde{C}u+u^T\tilde{W}+J'_{\boldsymbol{\lambda}}(\beta^0;u)\right\}=\boldsymbol{\hat{u}}.
\end{align*}
This shows weak convergence, which is not sufficient to conclude pattern convergence. The key to pattern convergence is to absorb the rest term $R_n(u)$ into the $O_p(1)$ stochastic term $W_n:=\nu_n\triangle(\cdot,\beta^0)$. We abbreviate the objective in (\ref{ultimate minimization problem}) by $F_n(W_n+R_n(u), u)$. The goal is to get rid of the dependency on $u$ in the first argument of $F_n$. To that end, let us fix $\varepsilon>0$. By uniform tightness of $\boldsymbol{\hat{u}}_n$, there exists a compact set $K$ such that $\mathbb{P}[\boldsymbol{\hat{u}}_n\in K^{\circ}]>1-\varepsilon$ for all $n\in\mathbb{N}$, where $K^{\circ}$ denotes the interior of $K$. Moreover, by Lemma \ref{Jonas Authority lemma}, there exists a sequence $\delta_M\downarrow 0$ as $M\rightarrow \infty$, such that $\mathbb{P}[sup_{u\in K}\Vert R_n(u)\Vert<\delta_M]>1-\varepsilon$ for every $n\geq M$. Let $\mathfrak{p}\in\mathfrak{P}$ be a fixed pattern and consider $p_n:=\mathbb{P}[\mathbf{patt}(\boldsymbol{\hat{u}}_n)=\mathfrak{p}]$.
\begin{align}
    p_n &= \mathbb{P}\left[\mathbf{patt}(\underset{u\in\mathbb{R}^p}{\operatorname{argmin}}\hspace{0,1cm} F_n(W_n+R_n(u), u))=\mathfrak{p}\right]\nonumber\\
    &\leq\mathbb{P}\left[\mathbf{patt}(\underset{u\in K^{\circ}}{\operatorname{argmin}}\hspace{0,1cm} F_n(W_n+R_n(u), u))=\mathfrak{p},\boldsymbol{\hat{u}}_n\in K^{\circ}, \underset{u\in K}{\operatorname{sup}}\Vert R_n(u)\Vert<\delta_M\right]+2\varepsilon\nonumber\\
    &\leq\mathbb{P}\left[\exists\hspace{0,1cm}\delta\in B_{\delta_M}(0):\mathbf{patt}(\underset{u\in K^{\circ}}{\operatorname{argmin}}\hspace{0,1cm} F_n(W_n+\delta, u))=\mathfrak{p},\boldsymbol{\hat{u}}_n\in K^{\circ}\right]+2\varepsilon\nonumber\\
    &\leq\mathbb{P}\left[\exists\hspace{0,1cm}\delta\in B_{\delta_M}(0): W_n +\delta \in \tilde{C}D^{\mathfrak{p}}+\partial J_{\boldsymbol{\lambda}^n/\sqrt{n}}(\mathfrak{p}_0)\right]+2\varepsilon\label{transition into optimality}\\
    &\leq\mathbb{P}\left[ W_n\in \left(\tilde{C}D^{\mathfrak{p}}+\partial J_{\boldsymbol{\lambda}^n/\sqrt{n}}(\mathfrak{p}_0)\right)^{\delta_M}\right]+2\varepsilon\nonumber,
\end{align}
where in (\ref{transition into optimality}) we have used optimality condition $0\in\partial F_n(W_n+\delta, u)$ and denoted by $D^{\mathfrak{p}}$ the intersection of $\mathbf{patt}^{-1}(\mathfrak{p})$ with $K$ and $\mathfrak{p}_0=\mathbf{patt}_{\beta^0}(\mathfrak{p})$. Setting $B_n:=\tilde{C}D^{\mathfrak{p}}+\partial J_{\boldsymbol{\lambda}^n/\sqrt{n}}(\mathfrak{p}_0)$, we have $\tilde{C}D^{\mathfrak{p}}+\partial J_{\boldsymbol{\lambda}/\sqrt{n}}(\mathfrak{p}_0)\overset{d_H}{\longrightarrow} \tilde{C}D^{\mathfrak{p}}+\partial J_{\boldsymbol{\lambda}}(\mathfrak{p}_0)=:B$ converging in the Hausdorff distance. Consequently, also $B_n^{\delta_M}\overset{d_H}{\longrightarrow}B^{\delta_M}$ and applying Lemma \ref{Hausdorff lemma}:
%We have seen in proof of Theorem \ref{theorem pattern} if a sequence of convex sets $B_n$ converges in Hausdorff distance to a convex set $B$ and $W_n$ converges in distribution to a random vector $W$, which is absolutely continuous w.r.t. the Lebesque measure, then $\mathbb{P}[W_n\in B_n]\rightarrow\mathbb{P}[W\in B]$.
\begin{equation*}
    \limsup_{n\rightarrow\infty} p_n \leq \limsup_{n\rightarrow\infty}\mathbb{P}[W_n\in B_n^{\delta_M}]+2\varepsilon = \mathbb{P}[
    \tilde{W}\in B^{\delta_M}]+2\varepsilon \leq \mathbb{P}[\tilde{W}\in B] +3\varepsilon,
\end{equation*}
where in the last inequality we have taken $\delta_M$ small enough and used (\ref{B^delta approx}). Similarly, we can bound $p_n$ from below 
\begin{align}
    p_n &\geq\mathbb{P}\left[\mathbf{patt}(\underset{u\in K^{\circ}}{\operatorname{argmin}}\hspace{0,1cm} F_n(W_n+R_n(u), u))=\mathfrak{p},\boldsymbol{\hat{u}}_n\in K^{\circ}, \underset{u\in K}{\operatorname{sup}}\Vert R_n(u)\Vert<\delta_M\right]\nonumber\\
    &\geq\mathbb{P}\left[\forall\hspace{0,1cm}\delta\in B_{\delta_M}(0):\mathbf{patt}(\underset{u\in K^{\circ}}{\operatorname{argmin}}\hspace{0,1cm} F_n(W_n+\delta, u))=\mathfrak{p},\boldsymbol{\hat{u}}_n\in K^{\circ}, \underset{u\in K}{\operatorname{sup}}\Vert R_n(u)\Vert<\delta_M\right]\nonumber\\
    &\geq\mathbb{P}\left[\forall\hspace{0,1cm}\delta\in B_{\delta_M}(0): W_n +\delta \in \tilde{C}D^{\mathfrak{p}}+\partial J_{\boldsymbol{\lambda}^n/\sqrt{n}}(\mathfrak{p}_0)\right]-2\varepsilon\nonumber\\
    &=\mathbb{P}\left[ W_n\in \left(\tilde{C}D^{\mathfrak{p}}+\partial J_{\boldsymbol{\lambda}^n/\sqrt{n}}(\mathfrak{p}_0)\right)^{-\delta_M}\right]-2\varepsilon\nonumber,
\end{align}
Since $B_n^{-\delta_M}\overset{d_H}{\longrightarrow}B^{-\delta_M}$, by the same argument applying Lemma \ref{Hausdorff lemma} and (\ref{B^delta approx}):
\begin{equation*}
    \liminf_{n\rightarrow\infty} p_n \geq \liminf_{n\rightarrow\infty}\mathbb{P}[W_n\in B_n^{-\delta_M}]-2\varepsilon = \mathbb{P}[
    \tilde{W}\in B^{-\delta_M}]-2\varepsilon \geq \mathbb{P}[\tilde{W}\in B] -3\varepsilon,
\end{equation*}
As a result $\mathbf{patt}(\boldsymbol{\hat{u}}_n)$ converges in distribution to $\mathbf{patt}(\boldsymbol{\hat{u}})$. The remaining assertions are direct consequences of the weak pattern convergence of $\boldsymbol{\hat{u}}_n$.
\end{proof}

\begin{example}(Quadratic loss)
 Set $g(y_i, X_i, \beta)=(\varepsilon_i - t(\beta))^2/2$, with $t(\beta)=X^T_i(\beta-\beta^0)$. %, We will treat this as a special example of the Huber loss.
 We have $\nabla_{\beta}\hspace{0,1cm}g(y_i,X_i,\beta) = X_i(t(\beta)-\varepsilon_i) $, hence the first order term in (\ref{key expansion of g}) reads $\triangle(\cdot,\beta^0)=-\varepsilon_i X_i$ and by independence $C_\triangle = \mathbb{E}(\triangle(\cdot,\beta^0)\triangle(\cdot,\beta^0)^T) = \mathbb{E}[\varepsilon_i^2]\mathbb{E}[X_i X_i^T]=\sigma^2C$. Moreover, $\nabla^2_{\beta}\hspace{0,1cm}g(y_i,X_i,\beta)= X_i X_i^T$, which gives $\tilde{C}=\nabla^2\vert_{\beta=\beta^0}G(\beta)=\mathbb{E}[\nabla^2\vert_{\beta=\beta^0}\hspace{0,1cm}g(y_i,X_i,\beta)] = C$. This coincides with Theorem \ref{sqrt-asymptotic}.
\end{example}

\begin{example}(Huber loss \cite{hub67beh}) Consider $
    g(y_i, X_i, \beta)= H(\varepsilon_i - t(\beta))$, with $t(\beta) := X^T_i(\beta-\beta^0)$ and
the Huber loss $H:\mathbb{R}\rightarrow\mathbb{R}$, given by
\begin{align*}
    H(x)&:=k \left(-x-\dfrac{k}{2}\right)\mathbf{1}_{\{x<k\}}+\dfrac{1}{2}x^2\mathbf{1}_{\{\vert x \vert<k\}}+k\left(x-\dfrac{k}{2}\right)\mathbf{1}_{\{x>k\}}.
\end{align*}

We consider
\begin{align*}
    D(x,t):=\dfrac{\partial}{\partial t}H(x-t)&= k\mathbf{1}_{\{x-t<-k\}}-(x-t)\mathbf{1}_{\{\vert x-t\vert \leq k\}}-k\mathbf{1}_{\{x-t>k\}}\\
    \dfrac{\partial^2}{\partial^2 t} H(x-t)&= \mathbf{1}_{\{\vert x-t\vert \leq k\}}
\end{align*}
By chain rule $\nabla_{\beta}\hspace{0,1cm}g(y_i,X_i,\beta)=X_i D(\varepsilon_i,t(\beta))$, hence the $\triangle(\cdot,\beta^0)$ term in expansion (\ref{key expansion of g}) reads $\triangle(\varepsilon_i,X_i,\beta^0)=X_i D(\varepsilon_i,0)$. The covariance matrix of $\triangle(\cdot,\beta^0)$ equals $C_{\triangle}=\delta C$, where $C$ is the covariance matrix of $X_i$ and
\begin{equation}
    \delta=\mathbb{E}[D(\varepsilon_i,0)^2]=\mathbb{E}[\varepsilon_i^2\mathbf{1}_{\vert\varepsilon_i\vert<k}]+k^2\mathbb{P}[\vert\varepsilon_i\vert>k].
\end{equation}
Moreover, $\nabla^2_{\beta}\hspace{0,1cm}g(y_i,X_i,\beta)=X_iX_i^T\mathbf{1}_{\{\vert \varepsilon_i-t(\beta)\vert \leq k\}}$, and we get $\tilde{C}=\nabla^2\vert_{\beta=\beta^0}G(\beta)=\gamma C$, where $\gamma=\mathbb{P}[\vert\varepsilon_i\vert<k]$. In particular, letting $k\rightarrow\infty$, we get $\delta\rightarrow\mathbb{E}[\varepsilon_i^2]=\sigma^2$ and $\gamma\rightarrow 1$, and recover the standard quadratic case with $\tilde{C}=C$ and $C_{\triangle}=\sigma^2 C$. %(Note that for $k\rightarrow 0$ both $\delta$ and $\gamma$ go to zero, so the limiting covariance matrices are zero). 
 Further, for $C=\mathbb{I}_p$ we have $\tilde{C}=\gamma\mathbb{I}_p, \tilde{W}\sim\mathcal{N}(0,\delta\mathbb{I}_p)$ and for  $\lambda_i:=\sqrt{\delta}\lambda_{i}^{BH(q)}=\sqrt{\delta}\Phi^{-1}(1-iq/2p)$, FDR($\hat{\boldsymbol{\beta}}_n$) is asymptotically controlled at level $q(\vert I_0 \vert/p)$. (This can be verified by substituting $\tilde{u}=\sqrt{\gamma}u$ in (\ref{V(u) general}) and using Example \ref{example FDR control}; see Appendix \ref{appendix FDR}) Note that $\lambda_i$ do not depend on $\gamma$, but only on $\delta$. Because $\delta=\delta_k\leq\sigma^2$ for any $k$, it suffices to penalize with smaller penalties to obtain FDR control for the Huber loss than for the standard Quadratic loss. In particular, $\delta\rightarrow 0$ as $k\rightarrow 0$, so that $\lambda_i$ are significantly smaller for small $k$. Similar to Example \ref{example FDR control}, the FDR control holds more generally for a block diagonal $C=\mathbb{I}_{I_0}\oplus \Sigma$.
\end{example}
 Quantile regression \cite{koenker1978regression} is a popular method where one assumes that there is a linear relation between the fixed effect and a fixed  quantile of $y$. It has been studied earlier with $L_1$ regularization in \cite{l1normquantile}.
\begin{example} (Quantile) 
For the Quantile regression we drop the assumption of finite variance and zero mean of the noise $\varepsilon$. Instead
we assume that the CDF of $\varepsilon$ has a continuous density $f$ and CDF $F$, which satisfies $F^{-1}(\alpha)=0$, and $f(0)>0$.  Consider $
    g(y_i, X_i, \beta)= \vert\varepsilon_i - t(\beta)\vert_{\alpha}$, with $t(\beta) := X^T_i(\beta-\beta^0)$. The loss function $\vert \cdot \vert_{\alpha}$ is everywhere differentiable except at $0$. The weak derivative is given by $\vert x \vert_{\alpha}^{'}= - (1-\alpha)\mathbf{1}_{\{x \leq 0\}}+\alpha \mathbf{1}_{\{x > 0\}}$. Note that $t\mapsto\mathbb{E}[\vert\varepsilon_i-t\vert_{\alpha}]$ attains its minimum at the solution to $0=- (1-\alpha)\mathbb{P}(\varepsilon_i \leq t)+\alpha \mathbb{P}(\varepsilon_i > t)$, which is the $\alpha$ quantile of $\varepsilon_i$. Now since this minimum is obtained by setting $t=0$ it is clear that $G$ has a minimum value at $\beta^0$.

By chain rule $\nabla_{\beta}\hspace{0,1cm}g(y_i,X_i,\beta) = -X_i \vert \varepsilon_i - t(\beta) \vert_{\alpha}^{'}=X_i((1-\alpha)\mathbf{1}_{\{\varepsilon_i \leq t(\beta)\}}-\alpha \mathbf{1}_{\{\varepsilon_i > t(\beta)\}})$ and the $\triangle(\cdot,\beta^0)$ term in expansion (\ref{key expansion of g}) reads $\triangle(\varepsilon_i,X_i,\beta^0)=-X_i \vert \varepsilon_i\vert_{\alpha}^{'}$. Consequently, the covariance matrix of $\triangle(\cdot,\beta^0)$ equals $C_{\triangle}=\delta C$, where
\begin{equation*}
    \delta= \mathbb{E}\bigr[ (\vert \varepsilon_i\vert_{\alpha}^{'})^2\bigr]=  (1-\alpha)^2+\alpha^2
\end{equation*}

Moreover,
\begin{align*}
    \nabla_{\beta}\mathbb{E}\bigr[\hspace{0,1cm}g(y_i,X_i,\beta)\bigr]&=\mathbb{E}\bigr[\hspace{0,1cm}\nabla_{\beta}g(y_i,X_i,\beta)\bigr]\\
    &=\mathbb{E}\Bigr[X_i\mathbb{E}\bigr[(1-\alpha)\mathbf{1}_{\{\varepsilon_i \leq t(\beta)\}}-\alpha \mathbf{1}_{\{\varepsilon_i > t(\beta)\}}\vert X_i\bigr]\Bigr]\\
    &=\mathbb{E}\Bigr[X_i\bigr((1-\alpha)\mathbb{P}({\varepsilon_i \leq t(\beta)} \vert X_i)-\alpha(1-\mathbb{P}({\varepsilon_i \leq t(\beta)} \vert X_i))\bigr)\Bigr]\\
    &= \mathbb{E}\Bigr[X_i\bigr(\mathbb{P}({\varepsilon_i \leq t(\beta)} \vert X_i)-\alpha\bigr)\Bigr],
\end{align*}

Thus $\tilde{C}=\nabla^2\vert_{\beta=\beta^0}G(\beta)=f(0) C$. Also, for $C=\mathbb{I}_p$, and $\lambda_i:=\sqrt{\delta}\Phi^{-1}(1-iq/2p)$,  FDR($\hat{\boldsymbol{\beta}}_n$) is asymptotically controlled at level $q(\vert I_0 \vert/p)$. For the Median loss $\alpha=1/2$, we obtain $\lambda_i=\left(1/\sqrt{2}\right)\Phi^{-1}(1-iq/2p)$. Again, the result remains valid for a block diagonal $C=\mathbb{I}_{I_0}\oplus \Sigma$. %However, in order to use the FDR result one needs to estimate $f(0)$, similar to generate confidence interval for the non-regularized quantile regression. Methods for estimation of $f(0)$  are dicussed in \cite{koenker1994confidence}.
\end{example}

%All above sections can be copied to the Bernoulli template
%Discussion, Appendix, Acknowledgements and Funding needs to be copied separately
\section{Discussion}
In this article we proposed a general theoretical framework for the asymptotic analysis of the pattern recovery and FDR control by SLOPE and its extensions. We mainly concentrated on the extensions to robust loss of functions, which are very important from the practical perspective. However, our proof techniques can be generalized to obtain similar results for applications of SLOPE to other important Generalized Linear Models, like the logistic or multinomial regressions, or the Gaussian Graphical models.  Another point for future interest is the analysis of algorithms combining the SLOPE penalty with $L_1$ or $L_2$ penalties. Our preliminary investigation suggests that combining the SLOPE and $L_2$ penalty allows to obtain FDR control under a wider set of scenarios than for SLOPE itself and we consider the investigation of the properties of this method as a very promising direction of research. Finally, our techniques can be extended for the analysis of adaptive versions of SLOPE. In \cite{ABSLOPE} SLOBE, the first version of the adaptive SLOPE, was constructed. It was illustrated that SLOBE can control FDR when predictors are strongly correlated. We expect that our techniques can be useful for the analysis of SLOBE properties as well as for construction of new adaptive versions of SLOPE, which apart from FDR control will also enhance the pattern recovery properties of SLOPE.

The asymptotic results presented in this paper are concerned with the classical asymptotics, where the model dimension $p$ is fixed and $n$ diverges to infinity. Our analysis shows that even  this classical setup deriving the results on the pattern convergence requires the development of new tools and substantially more care as compared to the analysis of the convergence of the vector of parameter estimates. We believe that the our analysis framework based on the set Hausdorff distance can be extended for the analysis of the SLOPE properties in the high dimensional setup and consider our work as an important first step in this direction.

\section{Acknowledgments}
The authors acknowledge the support of the Swedish Research Council, grant no. 2020-05081.
We would like to thank Elvezio Ronchetti and Micha{\l} Kos for the discussions on the robust versions of SLOPE and further Piotr Graczyk, Bartosz Ko{\l}odziejek, Tomasz Skalski, Patrick Tardivel and Ulrike Schneider for their helpful comments and discussions.

\bibliographystyle{plain}
\newpage
\bibliography{citation}

\newpage

\begin{appendix}
\section{Appendix}

\subsection{Weak pattern convergence is not trivial}\label{counterexample pattern convergence}
Consider the penalty $f(x)=\max\{x_1^2,x_2\}$ on $\mathbb{R}^2$, and let $\widehat{\boldsymbol{\beta}}_f = \operatorname{argmin_{\beta}}{G_n(\beta)+f_n(\beta)}$, with $f_n = n^{1/2}f$ and $G_n$ as in (\ref{G_n}) for the standard quadratic loss. Replicating Theorem \ref{sqrt-asymptotic}, 
\begin{align*}
    \widehat{\boldsymbol{u}}_{f,n}=\sqrt{n}(\widehat{\boldsymbol{\beta}}_f-\beta^0)=&\operatorname{argmin}u^{T}C_n u/2-u^{T}W_n+ n^{1/2}[f(\beta^0+u/\sqrt{n})-f(\beta^0)]\\
    \overset{d}{\longrightarrow}&\operatorname{argmin}{u^{T}Cu/2-u^{T}W+ f'({\beta^0};u)}=:\widehat{\boldsymbol{u}}_{f}.
\end{align*} 
For $\beta^0=0$, we have $n^{1/2}[f(\beta^0+u/\sqrt{n})-f(\beta^0)]=\max\{n^{-1/2}u_1^2,u_2\}=:g_n(u)$, and on the half line $\mathcal{K}=\{u_1>0, u_2=0\}$; the subdifferential $\partial g_n(u)=(2u_1/\sqrt{n},0)^T$ is zero-dimensional. We obtain
\begin{equation*}
    \mathbb{P}\left[\widehat{\boldsymbol{u}}_{f,n}\in \mathcal{K}\right]=\mathbb{P}\left[W_n\in \left\{C_n \begin{pmatrix} u_1 \\ 0 \end{pmatrix}+\begin{pmatrix} 2u_1/\sqrt{n} \\ 0 \end{pmatrix}:u_1>0\right\}\right]=0\hspace{0,2cm}\forall n,
\end{equation*} 
provided $W_n$ is absolutely continuous w.r.t. the Lebesque measure.
Furthermore, from (\ref{directional derivative equation}) we get $f'(0;u)=\max\{\langle(0, 0),(u_1,u_2)\rangle, \langle(0,1),(u_1,u_2)\rangle\}=\max\{0, u_2\}$, hence on $\mathcal{K}$ the subdifferential $\partial f'(0;u)=con\{(0,0)^T,(0,1)^T\}$ is one-dimensional. We get
\begin{equation*}
    \mathbb{P}\left[\widehat{\boldsymbol{u}}_{f}\in \mathcal{K}\right]=\mathbb{P}\left[W\in \left\{C \begin{pmatrix} u_1 \\ 0 \end{pmatrix}+\partial f'(0;u):u_1>0\right\}\right]>0,
\end{equation*}
since $C_{11}>0$. In particular, the SLOPE pattern $\mathbf{patt}(\widehat{\boldsymbol{u}}_{f,n})$ does not converge weakly to $\mathbf{patt}(\widehat{\boldsymbol{u}}_{f})$, despite the weak convergence of $\widehat{\boldsymbol{u}}_{f,n}$ to $\widehat{\boldsymbol{u}}_{f}$. 

Observe that $\widehat{\boldsymbol{u}}_{f,n}$ puts positive mass on the parabola $\{u_2=n^{-1/2}u_1\}$, where $\partial g_n(u)$ is one-dimensional, whereas $\widehat{\boldsymbol{u}}_{f}$ puts positive mass on the tangential space of the parabola at $0$ given by $\{u_2=0\}$. More precisely, the Lebesque decompositions of $\widehat{\boldsymbol{u}}_{f,n}$ and $\widehat{\boldsymbol{u}}_{f}$ w.r.t. the Lebesque measure yield different singular sets; the (flattening) parabola and the x-axis respectively. This gives some intuition for why linearity of the functions in the penalty $f=\max\{f_1,..,f_N\}$ is essential for convergence on convex cones. By applying the Portmanteau Lemma, we get weak convergence of $I(\widehat{\boldsymbol{u}}_{f,n})$ to $I(\widehat{\boldsymbol{u}}_{f})$ w.r.t. the general pattern as described in \ref{General penalty and pattern}. % This gives some intuition for why linearity of the functions in the penalty $f=\max\{f_1,..,f_N\}$ is essential for convergence in the 

\subsection{Proximal operator}\label{proximal operator}
Similar to the SLOPE optimization problem, the minimization of $(\ref{V(u)})$ and $(\ref{V(u) general})$ can be solved using proximal methods. The proximal operator of $J'_{\boldsymbol{\lambda}}({\beta^0};u)$ is given by:
\begin{equation*}
    \text{prox}_{J_{\boldsymbol{\lambda},{\beta^0}}}(y):=\underset{u\in\mathbb{R}^p}{\operatorname{argmin }}\hspace{0.2cm} (1/2)\Vert u-y\Vert_2^2 + J'_{\boldsymbol{\lambda}}({\beta^0};u)
\end{equation*}

%For a vector $v\in\mathbb{R}^p$ and $I\subset\{1,..,p\}$, let $v_I\in\mathbb{R}^{ I}$ denote the vector of entries on $I$. Furthermore, let $\mathcal{I}(\beta^0)=\{I_0, I_1, ..,I_m\}$ be the partition of $\beta^0$ into the clusters of the same magnitude.

Let $\mathcal{I}(\beta^0)=\{I_0, I_1, ..,I_m\}$ be the partition of $\beta^0$ into the clusters of the same magnitude. The directional SLOPE derivative $J'_{\boldsymbol{\lambda}}({\beta^0};u)$ is separable:
\begin{equation*}
    J'_{\boldsymbol{\lambda}}({\beta^0};u) = J^{I_0}_{\boldsymbol{\lambda}}(u) + J^{I_1}_{\boldsymbol{\lambda},{\beta^0}}(u) + ... + J^{I_m}_{\boldsymbol{\lambda},{\beta^0}}(u),
\end{equation*} 
with
\begin{align*}
     J^{I_0}_{\boldsymbol{\lambda}}(u) & =\sum_{i\in I_0}\boldsymbol{\lambda}_{\pi(i)}\vert u_i\vert, \\
      J^{I_j}_{\boldsymbol{\lambda},{\beta^0}}(u) & =\sum_{i\in I_j}\boldsymbol{\lambda}_{\pi(i)}u_i sgn(\beta^0_i),
\end{align*}
where the permutation $\pi$ in $J'_{\boldsymbol{\lambda}}({\beta^0};u)$ sorts the limiting pattern of $u$ w.r.t. $\beta^0$, i.e; $\vert \mathfrak{p}_0 \vert_{\pi^{-1}(1)}\geq\dots\geq\vert \mathfrak{p}_0 \vert_{\pi^{-1}(p)}$, with $\mathfrak{p}_0=\mathbf{patt}_{\beta^0}(u)$.

 Hence 
 \begin{equation*}
     \text{prox}_{J_{\boldsymbol{\lambda},\beta^0}}(y)=\text{prox}_{J^{I_0}_{\boldsymbol{\lambda},{\beta^0}}}(y)\oplus\text{prox}_{J^{I_1}_{\boldsymbol{\lambda},{\beta^0}}}(y)\oplus\dots\oplus\text{prox}_{J^{I_m}_{\boldsymbol{\lambda},{\beta^0}}}(y)
 \end{equation*}
 
 Since we can treat each cluster separately, we can w.l.o.g. assume that $\beta^0$ consists of one cluster only. There are only two possible cases:
 
 In the first case, $\beta^0=0$ and the proximal operator is described in \cite{bogdan2015slope}:
\begin{align}
    \text{prox}_{J_{\boldsymbol{\lambda}}}(y)&=\underset{u\in\mathbb{R}^p}{\operatorname{argmin }}\hspace{0.2cm} (1/2)\Vert u-y\Vert_2^2 + J_{\boldsymbol{\lambda}}(u)\nonumber\\
   &=S_{y}\Pi_{y}\underset{\tilde{u}_1\geq\dots\geq\tilde{u}_p\geq0 }{\operatorname{argmin }}\hspace{0.2cm} (1/2)\Vert \tilde{u}-\vert y\vert_{(\cdot)}\Vert_2^2 + \sum_{i=1}^p\lambda_i\tilde{u}_i,\label{slope proximal operator optimization}
\end{align}
where $\vert y\vert_{(\cdot)}=\Pi^T_yS_{y}y$  arises by sorting the absolute values of $y$. (See Proposition 2.2 in \cite{bogdan2015slope} and notation in the section Subdifferential and Pattern.)

In the second case, $\beta^0$ consists of a single non zero cluster. In this case the penalty becomes $J'_{\boldsymbol{\lambda}}({\beta^0};u)=\sum_{i=1}^p \boldsymbol{\lambda}_{\pi(i)}(S_{\beta^0}u)_i=\sum_{i=1}^p \boldsymbol{\lambda}_{i}(S_{\beta^0}u)_{\pi^{-1}(i)}$, where $ (S_{\beta^0}u) _{\pi^{-1}(1)}\geq\dots\geq (S_{\beta^0}u) _{\pi^{-1}(p)}$. In particular, $J_{\boldsymbol{\lambda},\beta^0}(S_{\beta^0}u)=\sum_{i=1}^p \boldsymbol{\lambda}_{i}u_{\pi^{-1}(i)}=\sum_{i=1}^p \boldsymbol{\lambda}_{i}u_{(i)}$, with $ u _{(1)}\geq\dots\geq u _{(p)}$.

\begin{align}
    \text{prox}_{J_{\boldsymbol{\lambda},\beta^0}}(y)&=\underset{u\in\mathbb{R}^p}{\operatorname{argmin }}\hspace{0.2cm} (1/2)\Vert u-y\Vert_2^2 + J'_{\boldsymbol{\lambda}}({\beta^0};u)\nonumber\\
   & = S_{\beta^0}\underset{\tilde{u}\in\mathbb{R}^p }{\operatorname{argmin }}\hspace{0.2cm} (1/2)\Vert \tilde{u}- S_{\beta^0}y\Vert_2^2 + \sum_{i=1}^p\lambda_i\tilde{u}_{(i)}\nonumber\\
   & =  S_{\beta^0}\Pi\underset{\tilde{u}_1\geq\dots\geq\tilde{u}_p}{\operatorname{argmin }}\hspace{0.2cm} (1/2)\Vert \tilde{u}-(S_{\beta^0}y)_{(\cdot)}\Vert_2^2 + \sum_{i=1}^p\lambda_i\tilde{u}_i,\label{isotonic optimization}
\end{align}

where $(S_{\beta^0}y)_{(\cdot)}=\Pi ^T (S_{\beta^0}y)$ is the sorted \footnote{Note that the permutation matrix $\Pi$ depends both on $y$ and $\beta^0$. } $(S_{\beta^0}y)$ vector. The optimization problem in (\ref{isotonic optimization}) is very similar to the optimization problem in (\ref{slope proximal operator optimization}). The only difference is in the relaxed constraint, where the set of feasible solutions in (\ref{isotonic optimization}) allows for negative values. The optimization (\ref{isotonic optimization}) is a special case of the isotonic regression problem \cite{10.2307/2284712}:
\begin{align*}
\operatorname{minimize }\hspace{0.5cm} & \Vert x-z\Vert_2^2\nonumber \\
\text{subject to}\hspace{0.5cm} & x_1\geq\dots\geq x_p,
\end{align*}
where we set $z= (S_{\beta^0}y)_{(\cdot)} - \boldsymbol{\lambda} $.
\subsection{FDR constant}\label{appendix FDR} Let $Z\sim\mathcal{N}(0,\mathbb{I}_p)$ and $\gamma>0$. Consider the rejection procedure $\boldsymbol{\hat{\mathfrak{b}}}=\mathbf{patt}_{\beta^0}(\boldsymbol{\hat{u}})$, with 
\begin{align*}
    \boldsymbol{\hat{u}}&= \operatorname{argmin}_uu^{T}\gamma\mathbb{I}_p u/2-u^{T}\sigma Z+J'_{\boldsymbol{\lambda}}({\beta^0};u)\\
    &=\sqrt{\gamma}\operatorname{argmin}_{\tilde{u}}\tilde{u}^{T}\mathbb{I}_p \tilde{u}/2- \tilde{u}^{T}(\sigma/\sqrt{\gamma}) Z+J'_{\boldsymbol{(\lambda/\sqrt{\gamma}) }}({\beta^0};\tilde{u}),
\end{align*}
where $\tilde{u}=\sqrt{\gamma}u$. Therefore, if $\lambda_i=\sigma \lambda^{BH(q)}_{i}=\sigma \Phi^{-1}(1-iq/2p)$, we have $\text{FDR}(\boldsymbol{\hat{\mathfrak{b}}})\leq q(\vert I_0\vert /p)$, by Example \ref{example FDR control}. In particular, the penalties $\lambda_i$ do not depend on the constant $\gamma$.

\end{appendix}

\end{document}